\documentclass[10pt]{article}

\usepackage{amsmath, amsfonts, amsthm, amssymb}
\usepackage{mathtools}
\usepackage[noadjust]{cite}
\usepackage[scr=rsfso]{mathalfa}

\numberwithin{equation}{section}

\theoremstyle{remark}
\newtheorem*{remk}{Remark}

\theoremstyle{plain} 
\newtheorem{thm}{Theorem}[section]
\newtheorem{lem} [thm]{Lemma}
\newtheorem{prop}[thm]{Proposition}

\newtheorem{clm}{Claim}

\theoremstyle{definition}
\newtheorem{defn}{Definition}[section]

\usepackage[a4paper]{geometry}

\usepackage{xparse}

\providecommand{\keywords}[1] {\small {\emph{Keywords :}} #1}
\NewDocumentCommand{\norm}{m m} {\IfNoValueTF{#1} {\left\Vert{#2}\right\Vert} {\left\Vert{#2}\right\Vert_{#1}} }
\NewDocumentCommand{\subjclass}{o m} {\IfNoValueTF{#1}{\small {\emph{Mathematics Subject Classification :}}}{\small {\emph{#1 Mathematics Subject Classification :}}} #2}

\title{Global Well-posedness and Scattering for the Focusing Energy-critical Inhomogeneous Nonlinear Schr\"{o}dinger Equation with Non-radial Data}

\author{Dongjin Park}

\makeatletter
\renewcommand\@date{{
	\vspace{-\baselineskip}
	\vspace{1ex}
	{\normalsize \texttt{aseastar@kaist.ac.kr}}
	\\[1ex]
	{\normalsize Department of Mathematical Sciences,
	 Korea Advanced Institute of Science and Technology,
	 Daejeon 34141, Republic of Korea}
	}}
\makeatother

\begin{document}

\maketitle

\begin{abstract}
We consider the focusing energy-critical inhomogeneous nonlinear Schr\"{o}dinger equation
\[
iu_t + \Delta u = -|x|^{-b}|u|^{\alpha}u
\]
where $n \geq 3$, $0<b<\min(2, n/2)$, and $\alpha=(4-2b)/(n-2)$.
We prove the global well-posedness and scattering for every $n \geq 3$, $0<b<\min(2, n/2)$, and every non-radial initial data $\phi \in \dot{H}^1(\mathbb{R}^n)$ by the concentration compactness arguments of Kenig and Merle (2006) as in the work of Guzman and Murphy (2021). Lorentz spaces are adopted during the development of the stability theory in critical spaces as in Killip and Visan (2013).
The result improves multiple earlier works by providing a unified approach to the scattering problem, accepting both focusing and defocusing cases, and eliminating the radial-data assumption and additional restrictions to the range of $n$ and $b$.
\end{abstract}

\vspace{0.6em}

\noindent
\keywords{Inhomogeneous nonlinear Schr\"{o}dinger equation, scattering, stability theory, concentration-compactness principle, Lorentz spaces}

\noindent
\subjclass[2020]{35Q55, 35P25, 35B40}

\noindent
\emph{Declarations of interest :} None

\section{Introduction} \label{Section:1}

Let $n\geq 3$ be the dimension, $0<b<\min(2, n/2)$, $\mu = \pm 1$, and $\alpha > 0$.
We consider the following Cauchy problem of the \emph{inhomogeneous nonlinear Schr\"{o}dinger equation} (abbrev. \emph{INLS}).
\begin{equation} \label{INLS}
\left\{
\begin{aligned}
&iu_t + \Delta u = g(x)|u|^{\alpha}u \\
&u(t_0) = \phi \in \dot{H}^1(\mathbb{R}^n)
\end{aligned}
\right. \tag{INLS}
\end{equation}
where $g(x) = \mu |x|^{-b}$.
The equation has a background in physics as a mathematical model in the field of the Bose-Einstein condensation or the laser guiding with axially non-uniform plasma channels. (We refer to \cite{BelmonteBeitia200798}, \cite{YonggeunCho19}, \cite{FibichWang03}, \cite{Gill2000}, and \cite{TangShukla200776} for details.)

\vspace{0.6em}
The coefficient $g(x)$ in the nonlinear term of {INLS} represents the particle interactions or the electron density in the respective research field and lies in $C^N(\mathbb{R}^n \setminus \{0\})$ for some $N\in\mathbb{N}$
with growth or decay conditions as $|x| \to 0$ or $\infty$. (See \cite{YonggeunCho19}, \cite{FibichWang03}, \cite{Genoud07}, \cite{LiuWangWang06}, and \cite{Merle96} as examples.)
In this paper, we specifically consider $g(x) = \mu|x|^{-b}$, from which we can utilize both the scaling symmetry and the vanishing at infinity.
When $b = 0$, \eqref{INLS} reduces to the standard nonlinear Schr\"odinger equation (abbrev. \emph{NLS}).

\vspace{0.6em}
We start with recalling common terminologies in the field of NLS.
Given a solution $u : I \times \mathbb{R}^n \to \mathbb{C}$ to \eqref{INLS},
we refer to the time interval $I$ as its \emph{lifespan} and assume that $I$ contains the initial time $t = t_0$.
For the parameter $\mu$ of the nonlinearity sign,
we say that the equation \eqref{INLS} is \emph{defocusing} when $\mu = +1$ and \emph{focusing} when $\mu = -1$.
For the parameter $\alpha$ of the nonlinearity power,
we associate a value called the \emph{critical regularity} ${s_c} = n/2-(2-b)/\alpha$
as the following scaling transformation preserves both the equation \eqref{INLS} itself and the solution's $\dot{H}_x^{{s_c}}(\mathbb{R}^n)$ norm simultaneously.
\begin{equation*}
u(t,x) ~\mapsto~ u_\lambda(t,x) := \lambda^{\frac{n-2{s_c}}{2}}u(\lambda^2 t, \lambda x) \quad(\lambda>0)
\end{equation*}

\vspace{0.6em}
Provided that a solution $u$ of \eqref{INLS} lies in $C_t(I,H_x^1(\mathbb{R}^n))$,
we can observe the conservation of the following two quantities throughout the lifespan:
the \emph{mass} $M[u]$ and the \emph{energy} (or \emph{Hamiltonian}) $E[u]$.
\begin{align*}
M[u] = M[u(t)]
&:= \int_{\mathbb{R}^n} \big|u(t,x)\big|^2 dx \\
E[u] = E[u(t)]
&:=	\int_{\mathbb{R}^n} \frac{1}{2} \big|\nabla u(t,x)\big|^2 + \frac{\mu}{\alpha+2}|x|^{-b}|u(t,x)|^{\alpha+2} dx
\end{align*}
Three important cases $s_c = 0$, $s_c \in (0, 1)$, and $s_c = 1$ are named as \emph{mass-critical}, \emph{intercritical}, and \emph{energy-critical} respectively, which is based on the fact that the scaling transformation also preserves $M[u(t)]^{1-{s_c}} E[u(t)]^{{s_c}}$ during the lifespan.


\vspace{0.6em}
Before discussing the research history of INLS,
we would like to review the case of NLS.
The energy-critical regime is mostly concerned, but other regimes with notable methodologies will also be mentioned.

\vspace{0.6em}
We begin by reviewing the earlier results for the defocusing energy-critical NLS ($\mu=+1$, $b=0$, and $\alpha=\frac{4}{n-2}$).
Bourgain \cite{Bourgain99} addressed the scattering problem for $n=3$ under the radial assumption using induction on energy.
Tao \cite{Tao05} later expanded this result to $n\geq 3$ while simplifying the argument and refining the scattering size estimates.
For non-radial data, Colliander et al. \cite{CollianderKSTT08} tackled $n=3$ using energy induction and localized interaction Morawetz inequalities. Ryckman and Visan \cite{RyckmanVisan07} worked on $n=4$, refining earlier arguments, and Visan \cite{Visan07} extended the analysis to higher dimensions $n\geq 5$ using fractional chain rules and stability theory.

\vspace{0.6em}
Next, we turn to the focusing energy-critical NLS ($\mu = -1$. $b=0$, and $\alpha=\frac{4}{n-2}$).
Kenig and Merle \cite{KenigMerle} solved the scattering problem for
$n=3,4,5$ with radial data, pioneering the application of the concentration-compactness principle to the scattering theory.
For non-radial data, Killip and Visan \cite{KillipVisan10} extended this framework to
$n\geq 5$, and Dodson \cite{Dodson19} recently addressed $n=4$.
The case $n=3$ with non-radial data still needs to be solved, though Cheng, Guo, and Zheng \cite{ChengGuoZheng23} recently studied global weak solutions.

\vspace{0.6em}
The techniques developed for addressing the scattering problem in the energy-critical NLS have greatly influenced the studies of NLS with other critical regularities.
For the defocusing energy-supercritical NLS (\(\mu = +1\), \(b = 0\), and \(\alpha > \frac{4}{n-2}\)), Bulut \cite{Bulut23} demonstrated the scattering theory for \((n,\alpha) = (3,6)\) with radial data under the \(L_t^\infty \dot{H}_x^{7/6}\) boundedness assumption\,\footnote{If we drop the \(L_t^\infty \dot{H}_x^{s_c}\) boundedness assumption, then it is possible for the solution of the energy-supercritical NLS to blow up even with the smooth initial data.
See Merle et al. \cite{MerleRRS22}.}, following ideas from \cite{Tao05}.

In the focusing intercritical NLS (\(\mu = -1\), \(b = 0\), and \(\frac{4}{n} < \alpha < \frac{4}{n-2}\)), Holmer and Roudenko \cite{HolmerRoudenko08} analyzed the case \((n, \alpha) = (3,2)\) for radial data using the concentration-compactness principle, later improved by Duyckaerts, Holmer, and Roudenko \cite{Duyckaerts08} to include non-radial data.
Fang, Xie, and Cazenave \cite{FangXieCazenave11} further completed the scattering theory for all \(n \geq 1\), covering all the intercritical cases. Alternative methods for higher dimensions, such as the virial-Morawetz approach, were explored by Dodson and Murphy \cite{DodsonMurphyRad, DodsonMurphy18} and Dinh \cite{Dinh20u}. The concentration-compactness method has also been applied to other dispersive equations, such as the generalized Korteweg–de Vries equation by Farah et al. \cite{Farah-gKdV}, and the generalized Benjamin-Ono equation by Kim and Kwon \cite{KihyunKim-Kwon-gBO}.

\vspace{0.6em}
Turning to the INLS, we note that scattering theory for this equation has emerged only recently.
Many studies focus on the intercritical case (\(\frac{2(2-b)}{n} < \alpha < \frac{2(2-b)}{n-2}\)), where the singular coefficient \( |x|^{-b} \) is manageable due to the assumption of \(H^1(\mathbb{R}^n)\) initial data. Dinh \cite{Dinh2019} addressed the defocusing case for \(n \geq 4\) and partially for \(n = 3\), building on Visciglia's approach \cite{Visciglia09}.
For the focusing case, Cardoso et al. \cite{CardosoFGM20} handled \(n \geq 2\) and \(0 < b < \min(2, \frac{n}{2})\), extending the works of Farah and Guzmán \cite{FarahGuzman17, FarahGuzman20} and Miao et al. \cite{MiaoMurphyZhang19}.
In addition, Dinh and Keraani \cite{DinhKeraani21} refined the long-time dynamics for \(n \geq 1\) by building the criteria for global existence, scattering (\(n\geq 2\)), and finite time blow-up. Campos and Cardoso \cite{CamposCardoso} presented a virial-Morawetz approach inspired by Dodson and Murphy's work \cite{DodsonMurphyRad} for \(n \geq 3\).

\vspace{0.6em}
There are also studies on energy-critical INLS (\(\alpha = \frac{2(2-b)}{n-2}\)), where controlling the coefficient \( |x|^{-b} \) is more challenging. Cho, Hong, and Lee \cite{ChoHongLee20} resolved the focusing case for \(n = 3\) and \(0 < b < 4/3\) with radial data. Cho and Lee \cite{ChoLee21} extended \cite{ChoHongLee20} to \(n = 3\) and \(4/3 \leq b < 3/2\) using weighted Sobolev spaces, which overcomes the parameter restriction caused by Hardy inequalities within Lebesgue spaces.
Guzmán and Xu \cite{GuzmanXu24.6}, building on Guzmán and Murphy \cite{GuzmanMurphy21}, tackled the scattering theory to \(n=3,4,5\) and \(0 < b \leq \min(\frac{4}{n}, \frac{6-n}{2})\) for non-radial data.


\vspace{0.6em}
In the author's previous work \cite{Park24128202}, the scattering theory of the defocusing energy-critical ($\mu = +1$, $\alpha = 2(2-b)/(n-2)$) {INLS} has been made under the radial data assumption.
It takes the road map of Tao \cite{Tao05} as the main methodology and adopts Lorentz spaces for Strichartz estimates by inspiration from the local well-posedness results of Aloui and Tayachi \cite{AlouiTayachi}.
The result is stated as the following two theorems.

\begin{thm} [{\cite[Theorem 1.1]{Park24128202}}] \label{thrm:prev-work-main}
Let $n\geq 3$ be an integer and $b>0$ obey the following range restriction.
\begin{equation*}
\left\{
\begin{aligned}
&0<b<3/2 && \text{if } n=3, \\
&0<b<1 && \text{if } n=4, \\
&0<b<1/2 \text{\: or \:} 5/4<b<2 && \text{if } n=5, \\
&0<b<2 && \text{if } n\geq 6.
\end{aligned}
\right.
\end{equation*}
Let $[T_-, T_+]$ be a compact interval, and let $u$ be the unique spherically symmetric solution of \eqref{INLS} associated with the initial data $\phi \in H^1(\mathbb{R}^n)$ and lying in the solution space
\begin{equation*}
C_t([T_-, T_+], H_x^1(\mathbb{R}^n))
\cap L_t^{p_*}([T_-, T_+], W_x^{1;q_*,2}(\mathbb{R}^n))
\end{equation*}
where $p_* = 2(n+2)/(n-2+2b)$ and $q_* = 2n(n+2)/(n^2+4-4b)$. Then we have
\begin{equation*}
\norm{L_t^{\frac{2(n+2)}{n-2}}([T_-, T_+], L_x^{\frac{2(n+2)}{n-2}}(\mathbb{R}^n))}{u}
\leq	C\exp(CE^C)
\end{equation*}
where $E = E[\phi]$ is the solution's energy and $C$ are absolute constants depending only on $n$ and $b$.\\
In particular, if $n=3,4,5$ and $0<b<(6-n)/2$, the same estimate holds when we replace the inhomogeneous Sobolev-Lorentz spaces with homogeneous ones:
we can assume $\phi$ to only lie in $\dot{H}^1(\mathbb{R}^n)$,
in which case
the solution only lies in $C_t([T_-, T_+], \dot{H}_x^1(\mathbb{R}^n)) \cap L_t^{p_*}([T_-, T_+], \dot{W}_x^{1;q_*,2}(\mathbb{R}^n))$.
\end{thm}

\begin{thm} [{\cite[Theorem 1.2]{Park24128202}}] \label{cor:prev-work-scattering}
Let $n$, $b$, $p_*$, $q_*$ and $u$ be defined as in Theorem \ref{thrm:prev-work-main}. Then the solution $u$ is global and scatters in $\dot{H}_x^1(\mathbb{R}^n)$, in the sense that $u$ in fact belongs to
\begin{equation*}
C_t(\mathbb{R}, \dot{H}_x^1(\mathbb{R}^n))
\cap L_t^{p_*}(\mathbb{R}, \dot{W}_x^{1;q_*,2}(\mathbb{R}^n))
\end{equation*}
and there exist profiles $U_+,U_-\in \dot{H}^1(\mathbb{R}^n)$ such that
\begin{equation*}
\lim_{t\to+\infty} \norm{\dot{H}_x^{1}}{u(t) - e^{it\Delta}U_+}
=\lim_{t\to-\infty} \norm{\dot{H}_x^{1}}{u(t) - e^{it\Delta}U_-}
= 0.
\end{equation*}
\end{thm}

\cite{Park24128202} has the significance of having contributed to the large-data global well-posedness theory in the energy-critical regime of \eqref{INLS} for higher dimensions.
It also provides a quantitative bound on the scattering size of the solution,
an interesting property that also appears in \cite{Bourgain99}, \cite{Tao05}, and \cite{CollianderKSTT08}.

\vspace{0.6em}
On the other hand, \cite{Park24128202} also possesses some weak points.
We can notice that the method of \cite{Tao05} is not designed to apply to the non-radial data and the focusing case.
In addition, the analysis with the low nonlinearity power $\alpha < 1$ is much harder and sometimes even misses the result since the derivative of the nonlinearity is not Lipschitz. It often negatively impacts the proof of uniqueness in the local theory, especially when we consider critical spaces only.

\vspace{0.6em}
In this paper, we establish the global well-posedness and the scattering of the \emph{focusing energy-critical} ($\mu = -1$, $\alpha = 2(2-b)/(n-2)$) \emph{INLS} for every $n \geq 3$ and $0 < b < \min(2, n/2)$ without the radial data assumption.
We develop a stability theory and then derive an improved local well-posedness theory through the arguments of Killip and Visan \cite[Section 3]{KillipVisan13} and the use of Lorentz spaces.
The global well-posedness and the sub-threshold scattering can then be shown after understanding the concentration compactness arguments.
The new technique of Guzman and Murphy \cite{GuzmanMurphy21} specialized for INLS prevents the solution from escaping too far from the origin, which provides the compactness property for the minimal energy blow-up solution and essentially reduces the enemy scenarios down to the soliton.
The scattering in the defocusing energy-critical regime can be proved similarly after modifying the variational analysis, in which case the size threshold is removed, yet the proof is omitted.

\vspace{0.6em}
Before the statement of the main results,
we introduce the ground state $Q(x)$ for the focusing energy-critical case
as already stated in \cite{GuzmanXu24.6}.
It is a positive, radially decreasing, vanishing-at-infinity $C^\infty(\mathbb{R}^n \setminus \{0\})$ solution to the following elliptic equation, unique up to scaling.
See Yanagida \cite[Remark 2.1]{Yanagida} for details.
\[
Q(x) = \left(1 + \frac{|x|^{2-b}}{(n-2)(n-b)} \right) ^ {-\frac{n-2}{2-b}} \text{,\quad}
\Delta Q + |x|^{-b} Q^{\alpha + 1} = 0
\]

\vspace{0.6em}
The main result for the \emph{focusing} energy-critical INLS is as follows.
For the defocusing energy-critical INLS, we have only to remove the sub-threshold size assumption \eqref{mainthm-threshold} to obtain the full result.

\begin{thm} \label{mainthm} ~

Let $n\geq 3$ be an integer and $0<b<\min(2,n/2)$.
Then for every initial data $\phi \in \dot{H}^1(\mathbb{R}^n)$ such that
\begin{equation}
\label{mainthm-threshold}
E[\phi] < E[Q]  \text{\quad and\quad}  \norm{L_x^2}{\nabla \phi} < \norm{L_x^2}{\nabla Q},
\end{equation}
the solution $u \in C_t(\mathbb{R}, \dot{H}_x^1(\mathbb{R}^n))
\cap L_t^{2}(\mathbb{R}, \dot{W}_x^{1; \frac{2n}{n-2},2}(\mathbb{R}^n))$ of \eqref{INLS} exists globally in time, is unique, and scatters in $\dot{H}_x^1(\mathbb{R}^n)$.
In particular, the solution $u$ has finite scattering size
\begin{equation*}
\int_{\mathbb{R}} \int_{\mathbb{R}^n} |u|^{\frac{2(n+2)}{n-2}} dx\, dt \leq C(E[\phi]) < \infty
\end{equation*}
where $C(E[\phi])$ is a constant depending only on $n$, $b$, and $E[\phi]$, and continuous as a function of $E[\phi]$.
\end{thm}

The author would like to notice that a similar result to Theorem \ref{mainthm} has been obtained in an independent research by Liu and Zhang \cite{LiuZhang24-INLS}.
Theorem \ref{mainthm} is a generalization to several pre-existing results regarding the scattering theory of the energy-critical INLS, which include \cite{ChoHongLee20}, \cite{ChoLee21}, \cite{GuzmanMurphy21}, \cite{GuzmanXu24.6}, and \cite{Park24128202}.
In particular, we eliminate the radial data assumption and establishes a unified approach for the scattering theory on the range $n\geq 3$ and $0<b<\min(2,n/2)$.



\subsection*{Acknowledgement}

The author would like to express the deepest gratitude to his advisor Prof. Soonsik Kwon for his valuable advice and encouragement, which have greatly inspired this work.
The author would also like to appreciate Dr. Kiyeon Lee for several informative discussions.
This work has been partially supported by the National Research Foundation of
Korea (NRF-2019R1A5A1028324 and NRF-2022R1A2C1091499).

\vspace{2em}
\section{Preliminaries} \label{Section:2}

In this section, we prepare several necessary ingredients before diving into the concentration-compactness method.
We introduce Lorentz space with its variants and some basic properties and then use the fractional chain rules to establish the local well-posedness and the stability theory of INLS in the energy-critical regime.
This encompasses both the focusing and defocusing cases.
The cases of high ($\alpha \geq 1$) and low ($\alpha \leq 1$) nonlinearity powers are handled separately according to the argument of Killip and Visan \cite[Section 3]{KillipVisan13}.
In the last subsection, we write down the variational analysis in the focusing case.

\subsection{Function spaces} \label{Subsection:2.1}

In this subsection, we review the definition and properties of Lorentz spaces and their variations, as have been collected in Aloui and Tayachi \cite{AlouiTayachi} and in the author's previous work \cite{Park24128202}.
Again, we assume that all the functions appearing here are at least the tempered distributions to ensure that all the differential operators make sense.

\vspace{0.6em}
\begin{defn}
Let $1<p<\infty, ~1\leq q\leq\infty$. A function $f:\mathbb{R}^n \to \mathbb{C}$ is said to be in the \emph{Lorentz space} $L^{p,q}(\mathbb{R}^n)$ if its \emph{Lorentz quasinorm}
\begin{equation*}
\norm{L^{p,q}(\mathbb{R}^n)}{f}
:=	\left(\int_0^\infty \left( t^{-1/p}f^*(t) \right)^{q} \frac{dt}{t} \right)^{1/q}
\end{equation*}
is finite, where $f^*$ is the decreasing rearrangement of $f$.
\end{defn}

\vspace{0em}
\begin{defn}
Let $\gamma\in\mathbb{R}, ~1<p<\infty, ~1\leq q\leq\infty$. A function $f:\mathbb{R}^n \to \mathbb{C}$ is said to be in the \emph{homogeneous Sobolev-Lorentz space} $\dot{W}^{\gamma;p,q}(\mathbb{R}^n)$ if
\begin{equation*}
\norm{\dot{W}^{\gamma;p,q}(\mathbb{R}^n)}{f}
:=	\big\Vert |\nabla|^\gamma f \big\Vert_{L^{p,q}(\mathbb{R}^n)} < \infty,
\end{equation*}
or in the \emph{inhomogeneous Sobolev-Lorentz space} $W^{\gamma;p,q}(\mathbb{R}^n)$ if
\begin{equation*}
\norm{W^{\gamma;p,q}(\mathbb{R}^n)}{f}
:=	\big\Vert \langle\nabla\rangle^\gamma f \big\Vert_{L^{p,q}(\mathbb{R}^n)} < \infty
\end{equation*}
where $\langle x \rangle := (1+|x|^2)^{1/2}$ is the Japanese bracket.
\end{defn}

\vspace{0em}
\begin{defn}
Let $\gamma\in\mathbb{R}$, $2\leq p\leq\infty$, and $2\leq q<\infty$. A pair $(p,q)$ is said to be $\dot{H}^\gamma$-\emph{admissible} if $\dfrac{2}{p} + \dfrac{n}{q} = \dfrac{n}{2} - \gamma$,
and the \emph{Strichartz space} $\dot{S}^\gamma(I)$ refers to the set of all functions $u : I \times \mathbb{R}^n \to \mathbb{C}$ whose \emph{Strichartz norms} defined as below are finite. (Usual modification if $p = \infty$.)
\begin{align*}
\left\Vert u \right\Vert_{\dot{S}^{\gamma}(I)}
&=	\sup_{\text{$L^2$-admissible }(p,q)} \left\Vert u \right\Vert_{L_t^p(I,\dot{W}_x^{\gamma;q,2}(\mathbb{R}^n))}
=	\sup_{\text{$L^2$-admissible }(p,q)} \left(\int_I \big\Vert |\nabla|^\gamma u(t) \big\Vert_{L_x^{q,2}(\mathbb{R}^n)}^p \,dt \right)^{1/p}
\end{align*}
For this paper, we also define the \emph{dual Strichartz space} $\dot{N}^\gamma(I)$ as the set of all functions $u : I \times \mathbb{R}^n \to \mathbb{C}$ whose \emph{dual Strichartz norms} defined as below are finite, where $p' = p/(p-1)$ and $q' = q/(q-1)$. Notice that $\dot{N}^0(I)$ serves as the dual space of $\dot{S}^0(I)$.
\begin{align*}
\left\Vert u \right\Vert_{\dot{N}^{\gamma}(I)}
&=	\inf_{\text{$L^2$-admissible }(p,q)} \left\Vert u \right\Vert_{L_t^{p'}(I,\dot{W}_x^{\gamma;q',2}(\mathbb{R}^n))}
=	\inf_{\text{$L^2$-admissible }(p,q)} \left(\int_I \big\Vert |\nabla|^\gamma u(t) \big\Vert_{L_x^{q',2}(\mathbb{R}^n)}^{p'} \,dt \right)^{1/{p'}}
\end{align*}
\end{defn}

\vspace{0.6em}
Lorentz spaces share several useful norm inequalities with Lebesgue spaces, except for possibly worse proportionality constants. We list a couple of them here without proof.

\vspace{0.6em}
\begin{prop}[H\"{o}lder's inequality; {\cite[Proposition 2.3]{Lemarie02}, \cite[Theorem 3.4]{ONeil63}}]
Let $(X,\mu)$ be a $\sigma$-finite measure space, $1<p_1,p_2<\infty$, and $1\leq s_1,s_2\leq \infty$ such that
\begin{equation*}
1<p=\dfrac{1}{1/p_1+1/p_2}<\infty, \quad 1\leq s=\dfrac{1}{1/s_1+1/s_2} \leq\infty.
\end{equation*}
Then for every $f \in L^{p_1,s_1}(X)$ and $g \in L^{p_2,s_2}(X)$, we have
\begin{equation*}
\norm{L^{p,s}(X)}{fg}
\leq C \norm{L^{p_1,s_1}(X)}{f} \norm{L^{p_2,s_2}(X)}{g}
\end{equation*}
where $C$ is an absolute constant depending only on $p_1$, $p_2$, $s_1$, and $s_2$.
\end{prop}

\vspace{0em}
\begin{prop}[Sobolev embedding; {\cite[Theorem 2.4]{Lemarie02}}]
Let $0<\alpha<n$, $1<p<\tilde{p}<\infty$, and $1\leq s \leq\infty$ such that $-1/p + 1/\tilde{p} + \alpha/n = 0$. Then for every $f \in \dot{W}^{\alpha;p,s}(\mathbb{R}^n)$, we have
\begin{equation*}
\norm{L^{\tilde{p},s}(\mathbb{R}^n)}{f}
\leq C \norm{\dot{W}^{\alpha;p,s}(\mathbb{R}^n)}{f}
\end{equation*}
where $C$ is an absolute constant depending only on $p$, $s$, $n$, and $\alpha$.
\end{prop}

\vspace{0em}
\begin{prop}[Equivalent characterization of Sobolev-Lorentz spaces]~

Let $\gamma$ be a positive integer, $1<p<\infty$, and $1\leq q < \infty$. Then for every Schwartz function $f$ on $\mathbb{R}^n$,
\begin{gather*}
c { \norm{\dot{W}^{\gamma;p,q}(\mathbb{R}^n)}{f} } \leq { \norm{L^{p,q}(\mathbb{R}^n)}{\nabla^\gamma f} } \leq C { \norm{\dot{W}^{\gamma;p,q}(\mathbb{R}^n)}{f} }, \\
c { \norm{W^{\gamma;p,q}(\mathbb{R}^n)}{f} } \leq { \norm{L^{p,q}(\mathbb{R}^n)}{f} + \norm{L^{p,q}(\mathbb{R}^n)}{\nabla^\gamma f} } \leq C { \norm{W^{\gamma;p,q}(\mathbb{R}^n)}{f} }
\end{gather*}
where the constants $0 < c < C < \infty$ depend only on $\gamma, p, q, n$.
\end{prop}

\vspace{0em}
\begin{prop}[Strichartz estimates; {\cite{KeelTao}}]~ \label{strichartz}
\begin{itemize}
\item[(i)]
Let $2<p<\infty$. Then for every $f \in L^{p'}(\mathbb{R}^n)$ and $t\ne 0$, we have the dispersive estimate
\begin{equation*}
\norm{L_x^{p,2}(\mathbb{R}^n)}{e^{it\Delta}f}
\leq C |t|^{-n(\frac{1}{2} - \frac{1}{p})} \norm{L^{p',2}(\mathbb{R}^n)}{f}
\end{equation*}
where $1/p' + 1/p = 1$.
\item[(ii)]
Let $(p,q) \in [2,\infty]\times[2,\infty)$ be an $L^2$-admissible pair. Then for every $f \in L^{2}(\mathbb{R}^n)$, we have
\begin{equation*}
\norm{L_t^p(\mathbb{R}, L_x^{q,2}(\mathbb{R}^n))}{e^{it\Delta}f}
\leq C \norm{L^{2}(\mathbb{R}^n)}{f}
\end{equation*}
where $C$ is an absolute constant depending only on $p$, $q$, and $n$.
If $(\tilde{p},\tilde{q})$ is another $L^2$-admissible pair and $F \in L_t^{p'}(\mathbb{R}, L_x^{q',2}(\mathbb{R}^n))$, then we also have
\begin{equation*}
\norm{L_t^{\tilde{p}}(\mathbb{R}, L_x^{\tilde{q},2}(\mathbb{R}^n))}{\int_{0}^{t} e^{i(t-s)\Delta}F(s) ds}
\leq C \norm{L_t^{p'}(\mathbb{R}, L_x^{q',2}(\mathbb{R}^n))}{F}
\end{equation*}
where $C$ is an absolute constant depending only on $p$, $\tilde{p}$, $q$, $\tilde{q}$, and $n$.
\end{itemize}
\end{prop}

\subsection{Local well-posedness and stability} \label{Subsection:2.2}

In this subsection, we recall the local theory in $H^1$ for the energy-critical \eqref{INLS}, discuss the fractional chain rules adapted to \eqref{INLS},
and then establish the corresponding long-time perturbation theory to improve the local theory to also hold with respect to $\dot{H}^1$ by density.

\vspace{0.6em}
We start by recalling the following local well-posedness theorem based on Lorentz spaces, which is obtainable by the ordinary contraction arguments. See \cite{AlouiTayachi} and \cite{Park24128202} for details.
The initial time is set to $t=0$ in the statement.

\begin{thm}[Local well-posedness in $H^1$; {\cite[Theorem 2.5]{Park24128202}}] \label{thrm:2.5-well-posedness}~

Let $n\geq 3$, $0<b<\min(2,n/2)$.
Define the four exponents $p_0, p_b, q_0, q_b \in [1,\infty)$ by
\begin{equation*}
p_\beta = \frac{2(n+2)}{n-2+2\beta}, \quad
q_\beta = \frac{2n(n+2)}{n^2+4-4\beta} \quad (\beta = 0,b).
\end{equation*}
Then for every initial data $\phi \in H^1(\mathbb{R}^n)$, there exists a unique solution $u$ of \eqref{INLS} with the maximal lifespan $T_*(\phi)>0$ in the solution space
\begin{equation*}
C_t([0, T_*(\phi)), H_x^1(\mathbb{R}^n)) \cap L_t^{p_b}([0, T_*(\phi)), W_x^{1;q_b,2}(\mathbb{R}^n)).
\end{equation*}
In addition, we have the following properties.

\begin{itemize}
\item[(i)]
$u$ is unique in $L_t^{p_0}([0, T], W_x^{1;q_0,2}(\mathbb{R}^n)) \cap L_t^{p_b}([0, T], W_x^{1;q_b,2}(\mathbb{R}^n))$ for $0<T<T_*(\phi)$.

\item[(ii)]
$u \in L_t^p([0, T], W_x^{1;q,2}(\mathbb{R}^n))$ for $0<T<T_*(\phi)$ and every $L^2$-admissible pair $(p,q)$.

\item[(iii)]
If $\norm{\dot{H}_x^1}{\phi}$ is small enough, then $T_*(\phi) = \infty$ and $u$ lies in $L_t^p([0, \infty), W_x^{1;q,2}(\mathbb{R}^n))$ for every $L^2$-admissible pair $(p,q)$.
Moreover, $u$ scatters in $H_x^1(\mathbb{R}^n)$ forward in time, in the sense that $\lim\limits_{t\to\infty} e^{-it\Delta}u(t)$ exists in $H_x^1(\mathbb{R}^n)$.

\item[(iv)]
If $\alpha \geq 1$, then for every $0<T<T_*(\phi)$, there is $\delta_0 > 0$ such that if $\psi \in H_x^1(\mathbb{R}^n)$ with $\norm{H_x^1}{\phi-\psi} < \delta_0$, we have $T_*(\psi) > T$ and for every $L^2$-admissible pair $(p,q)$, the corresponding maximal solution $v$ satisfies
\begin{equation*}
\norm{L_t^{p}([0, T], W_x^{1;q,2}(\mathbb{R}^n))} {u - v}
\leq C \norm{H_x^1(\mathbb{R}^n)} {\phi - \psi}.
\end{equation*}

\item[(v)]
If $\alpha<1$, and if $\phi_j \to \phi$ in $H_x^1(\mathbb{R}^n)$,
then for every $0<T<T_*(\phi)$,
we have $u_j \to u$ in $C_t([0, T], H_x^1(\mathbb{R}^n))$,
where $u_j$ is the solution of \eqref{INLS} with initial data $\phi_j$.
\end{itemize}
\end{thm}

\vspace{0.6em}
As in the remark after {\cite[Theorem 2.5]{Park24128202}}, the low nonlinearity power $\alpha < 1$ poses a major concern since the derivative of the nonlinearity is not Lipschitz, which prevents us from the proof of uniqueness with the usual contraction arguments.
In the case of standard NLS, this issue was detected for $n > 6$, and thus Tao and Visan \cite{TaoVisan05} developed the dedicated perturbation theories using Besov-type function spaces.
A later work by Killip and Visan \cite{KillipVisan13} in 2013 has only used Sobolev spaces, which is more familiar, in conjunction with the fractional chain rule for fractional nonlinearities by \cite{Visan07}.

\vspace{0.6em}
We now aim to construct a perturbation theory following the arguments of \cite{KillipVisan13}.
To achieve that, we visit the fractional chain rules by Christ and Weinstein \cite{ChristWeinstein91} ($C^1$ nonlinearity) and Visan \cite{Visan07} (fractional nonlinearity) in the Lorentz space version as follows.

\begin{prop}[Fractional chain rule]~ \label{fract-chain}

Let $F$ be a $C^1(\mathbb{C})$ function.
Then for every Schwartz function $u$ on $\mathbb{R}^n$ and $0 < s < 1$, we have
\begin{equation*}
\norm{L^{p,q}(\mathbb{R}^n)}{|\nabla|^s F(u)}
\lesssim
\norm{L^{p_1,q_1}(\mathbb{R}^n)} {F'(u)}
\norm{L^{p_2,q_2}(\mathbb{R}^n)} {|\nabla|^{s} u}
\end{equation*}
provided that $1 < p_1,p_2,p < \infty$ and $1 \leq q_1,q_2,q < \infty$ (allowing $q = \infty$) satisfy
\begin{equation*}
0 < \frac{1}{p} = \frac{1}{p_1} + \frac{1}{p_2} < 1, \quad
0 \leq \frac{1}{q} \leq \frac{1}{q_1} + \frac{1}{q_2} \leq 1.
\end{equation*}
\end{prop}

\begin{proof}
The case $(p_1, p_2, p) = (q_1, q_2, q)$ corresponds to Christ and Weinstein \cite[Proposition 3.1]{ChristWeinstein91}.
For the general cases, see \cite[Lemma 2.4]{AlouiTayachi}.
\end{proof}

\begin{prop}[Fractional chain rule for the fractional nonlinearity]~ \label{fract-chain-holder}

Let $F$ be a H\"older-continuous function (i.e. $C^{0,\alpha}(\mathbb{C})$) of order $0 < \alpha \leq 1$.
Then for every Schwartz function $u$ on $\mathbb{R}^n$, $0 < s < \alpha$, and $s/\alpha < s_0 < 1$, we have
\begin{equation*}
\norm{L^{p,q}(\mathbb{R}^n)}{|\nabla|^s F(u)}
\lesssim
\norm{L^{p_1,q_1}(\mathbb{R}^n)} {u} ^{\alpha-{s}/{s_0}}
\norm{L^{p_2,q_2}(\mathbb{R}^n)} {|\nabla|^{s_0} u} ^{{s}/{s_0}}
\end{equation*}
provided that $1 < p_1,p_2,p < \infty$ and $1 \leq q_1,q_2,q < \infty$ (allowing $q = \infty$) satisfy
\begin{equation*}
0 < \frac{1}{p} = \frac{\alpha - s/s_0}{p_1} + \frac{s/s_0}{p_2} < 1, \quad
0 \leq \frac{1}{q} \leq \frac{\alpha - s/s_0}{q_1} + \frac{s/s_0}{q_2} \leq 1.
\end{equation*}
\end{prop}

\begin{proof}
The case $(p_1, p_2, p) = (q_1, q_2, q)$, $\alpha \in (0, 1)$ corresponds to Visan \cite[Proposition A.1]{Visan07}.
For the general case,
we observe the following equivalent characterization of Sobolev-Lorentz spaces
\begin{equation*}
\norm{L^{p,q}(\mathbb{R}^n)}{|\nabla|^s f} \sim \norm{L^{p,q}(\mathbb{R}^n)}{\mathscr{D}_s(f)}
\end{equation*}
for every $1 < p < \infty$ and $1\leq q < \infty$, where
\begin{equation}
\label{visan-fftc-Ds-oper}
\mathscr{D}_s(f)(x) = \bigg( \int_{0}^{\infty} \Big({\int_{|y|<1} |f(x+ry) - f(x)| dy}\Big)^2 \frac{dr}{r^{1+2s}} \bigg)^{1/2}.
\end{equation}
We can show the equivalence by appropriate real interpolations within the proof of Strichartz \cite[Subsection 2.3]{Strichartz67}.
The result then follows by applying H\"older's inequality to a pointwise estimate of $\mathscr{D}_s(f)$ as in the proof of \cite[Proposition A.1]{Visan07}.
The Lipscitz case $\alpha = 1$ can be proved similarly.
\end{proof}

\vspace{0.6em}

To make use of Propositions \ref{fract-chain} and \ref{fract-chain-holder}, we consider the following lemma of an interpolative estimate which takes care of the singular coefficient $|x|^{-b}$ while preserving the regularity. Later in the perturbation theory for \eqref{INLS}, when $F = |u|^\alpha$ or $|u|^{\alpha-2} u^2$ with $\alpha \geq 1$, Lemma \ref{Ds-Xmb-Dms} can be combined with Proposition \ref{fract-chain}. When $\alpha < 1$, it can be combined with Proposition \ref{fract-chain-holder} instead.

\begin{lem} \label{Ds-Xmb-Dms}
Let $0 \leq s \leq 1$ and $0<b<2$. Then for every Schwartz function $F$ on $\mathbb{R}^n$, we have
\begin{equation*}
\norm{L^{p,q}(\mathbb{R}^n)}{|\nabla|^s (|x|^{-b} F)}
\lesssim
\norm{L^{\frac{1}{\frac{1}{p} - \frac{b}{n}}, q}(\mathbb{R}^n)} {|\nabla|^{s} F}
\end{equation*}
provided that $1 < p < n/(b+s)$ and $1 \leq q < \infty$.
\end{lem}

\begin{proof}
The problem is equivalent to showing the following boundedness of a linear operator.
\begin{equation*}
\norm{L^{p,q}(\mathbb{R}^n)}{|\nabla|^s (|x|^{-b} |\nabla|^{-s} F)}
\lesssim
\norm{L^{\frac{1}{\frac{1}{p} - \frac{b}{n}}, q}(\mathbb{R}^n)} {F}
\end{equation*}

The cases $s = 0, 1$ are relatively easy. Note the $L^{p,q}(\mathbb{R}^n)$ boundedness of Riesz operators when $s = 1$. We now fill in the gap $0<s<1$.
By the real interpolation on H\"ormander-Mikhlin multiplier theorem, for every $(j,t) \in \{0,1\} \times \mathbb{R}$, we see that
\begin{equation*}
\norm{L^{p,q}(\mathbb{R}^n)}{|\nabla|^{j+it} (|x|^{-b} |\nabla|^{-(j+it)} F)}
\lesssim
\langle t \rangle ^{2n+4}
\norm{L^{\frac{1}{\frac{1}{p} - \frac{b}{n}}, q}(\mathbb{R}^n)} {F}.
\end{equation*}
The proportionality constants grow subexponentially in $|t|$ for both $j$, and the claim follows by Stein complex interpolation. See Bergh and L\"ofstr\"om \cite{BerghLofstrom76} for the complex interpolation of Lorentz spaces.
\end{proof}

\vspace{0.6em}
Combining Lemma \ref{Ds-Xmb-Dms} with each of Propositions \ref{fract-chain} and \ref{fract-chain-holder}, we obtain the following two inequalities. The proofs are omitted.

\begin{prop}[Fractional chain rule, $|x|^{-b}$-weighted]~ \label{fract-chain-inhomog}

Let $F$ be a $C^1(\mathbb{C})$ function.
Then for every Schwartz function $u$ on $\mathbb{R}^n$ and $0 < s < 1$, we have
\begin{equation*}
\norm{L^{p,q}(\mathbb{R}^n)}{|\nabla|^s (|x|^{-b} F(u))}
\lesssim
\norm{L^{p_1,q_1}(\mathbb{R}^n)} {F'(u)}
\norm{L^{p_2,q_2}(\mathbb{R}^n)} {|\nabla|^{s} u}
\end{equation*}
provided that $1 < p_1,p_2,p < \infty$ and $1 \leq q_1,q_2,q < \infty$ (allowing $q = \infty$) satisfy
\begin{equation*}
\frac{b+s}{n} < \frac{1}{p} = \frac{1}{p_1} + \frac{1}{p_2} + \frac{b}{n} < 1, \quad
0 \leq \frac{1}{q} \leq \frac{1}{q_1} + \frac{1}{q_2} \leq 1.
\end{equation*}
\end{prop}

\begin{prop}[Fractional chain rule for the fractional nonlinearity, $|x|^{-b}$-weighted]~ \label{fract-chain-holder-inhomog}

Let $F$ be a H\"older continuous function of order $0 < \alpha \leq 1$.
Then for every Schwartz function $u$ on $\mathbb{R}^n$, $0 < s < \alpha$, and $s/\alpha < s_0 < 1$, we have
\begin{equation*}
\norm{L^{p,q}(\mathbb{R}^n)}{|\nabla|^s (|x|^{-b} F(u))}
\lesssim
\norm{L^{p_1,q_1}(\mathbb{R}^n)} {u} ^{\alpha-\frac{s}{s_0}}
\norm{L^{p_2,q_2}(\mathbb{R}^n)} {|\nabla|^{s_0} u} ^{\frac{s}{s_0}}
\end{equation*}
provided that $1 < p_1,p_2,p < \infty$ and $1 \leq q_1,q_2,q < \infty$ (allowing $q = \infty$) satisfy
\begin{equation*}
\frac{b+s}{n} < \frac{1}{p} = \frac{\alpha - s/s_0}{p_1} + \frac{s/s_0}{p_2} + \frac{b}{n} < 1, \quad
0 \leq \frac{1}{q} \leq \frac{\alpha - s/s_0}{q_1} + \frac{s/s_0}{q_2} \leq 1.
\end{equation*}
\end{prop}

\begin{remk}
When $b < n(1-\alpha)$ in addition, Proposition \ref{fract-chain-holder-inhomog} is also provable by direct computation of a pointwise inequality for $\mathscr{D}_s (|x|^{-b} F(u))$, where $F(u)$ is either $|u|^\alpha$ or $|u|^{\alpha-2} u^2$ and the operator $\mathscr{D}_s$ is defined as \eqref{visan-fftc-Ds-oper}.
The combination of Lemma \ref{Ds-Xmb-Dms} and Proposition \ref{fract-chain-holder}, however, gives a more efficient result by not requiring the condition $b < n(1-\alpha)$.
\end{remk}

\vspace{0.6em}
We are now ready to develop the stability theory of \eqref{INLS}.
We consider the two cases separately: the high nonlinearity power ($\alpha > 1$) and the low nonlinearity power $\alpha \leq 1$.

\subsubsection{Case 1: high nonlinearity power ($\alpha \geq 1$)}

When $\alpha \geq 1$, Guzman and Xu \cite[Subsection 2.1]{GuzmanXu24.6} has already developed the stability theory of INLS for $n=3,4,5$ and $0 < b \leq \min(\frac{4}{n}, \frac{6-n}{2})$.
The restriction $b \leq \frac{4}{n}$ comes from the use of Hardy inequality, which is eliminable by detouring to the use of Lorentz spaces.
In this case, \cite[Lemma 2.1]{GuzmanXu24.6} turns into a simple application of H\"older inequalities:
\[
\norm{L_t^2(I, L_x^{\frac{2n}{n+2},2})} {|x|^{-b}|f|^\alpha g} \\
\leq
\norm{S(I)}{f}^\alpha \norm{\tilde{W}(I)}{g}
\]
where $S(I) = L_t^{\frac{2(n+2)}{n-2}}(I, L_x^{\frac{2(n+2)}{n-2}})$ and $
\tilde{W}(I) = L_t^{\frac{2(n+2)}{n-2+2b}}(I, L_x^{\frac{2(n+2)}{n^2+4-4b}, 2})
$.

\vspace{0.6em}
After a series of direct computations, we reach the stability theory for $\alpha \geq 1$ as the following lemma.
For the statement, we temporarily denote
\[
\norm{X(I)}{f}
=	\norm{L_{t}^{\frac{2(n+2)}{n-2}}(I,L_{x}^{\frac{2(n+2)}{n-2},2})}{f}
+	\norm{\tilde{W}(I)}{f}.
\]

\begin{lem}[Stability for INLS, $\alpha \geq 1$]~ \label{longperturb..a>1}

Let $n = 3, 4, 5$ and $0 < b \leq (6-n)/2$.
Let $t_0 \in I \subset \mathbb{R}$ be a compact time interval and
$\tilde{u} : I \times \mathbb{R}^n \to \mathbb{C}$ be an approximate INLS solution in the sense that
\begin{equation*}
i\tilde{u}_t + \Delta\tilde{u} = \mu |x|^{-b} |\tilde{u}|^\alpha \tilde{u} + e,
\end{equation*}
with the initial data $\tilde{u}(t_0) \in \dot{H}_x^1(\mathbb{R}^n)$ at time $t=t_0$, satisfying
\begin{equation*}
\norm{L_t^\infty(I, \dot{H}_x^1)}{\tilde{u}} \leq E, \quad
\norm{S(I)}{\tilde{u}} \leq L
\end{equation*}
for some constants $E, L > 0$.
Let $u(t_0) \in \dot{H}_x^1 (\mathbb{R}^n)$ such that
\begin{align*}
\norm{\dot{H}_x^1}{u(t_0) - \tilde{u}(t_0)} \leq D, \quad
\norm{X(I)} {e^{i(t-t_0)\Delta}(u(t_0) - \tilde{u}(t_0))} \leq \epsilon, \quad
\norm{L_t^2(I, \dot{W}_x^{1; \frac{2n}{n+2}, 2})}{e} \leq \epsilon
\end{align*}
for some constant $D>0$ and some $0 < \epsilon < \epsilon_1 = \epsilon_1(E,L,D)$.
Then, there exists a unique solution $u : I\times\mathbb{R}^n \to \mathbb{C}$ of \eqref{INLS} with the initial data $u(t_0)$ at time $t = t_0$ satisfying
\begin{align*}
\norm{X(I)} {u-\tilde{u}} \leq C(E,L,D)\epsilon, \quad
\norm{S(I)} {u} \leq C(E,L,D).
\end{align*}

\end{lem}
Lemma \ref{longperturb..a>1} is intended as the Lorentz-space variant of \cite[Proposition 2.4]{GuzmanXu24.6}.
The proof is omitted since it is replicable from \cite{GuzmanXu24.6}
with a few modification of function spaces: $X(I)$ and $\tilde{W}(I)$ take the roles of $S^1(I)$ and $W(I)$ respectively in \cite{GuzmanXu24.6}.

\subsubsection{Case 2: low nonlinearity power ($\alpha \leq 1$)}

When \(\alpha < 1\), although the stability theory for the energy-critical INLS has not been fully developed, some progress has been made for the standard energy-critical NLS, notably by Tao and Visan \cite{TaoVisan05} and Killip and Visan \cite{KillipVisan13}. Given the relevance of weighted fractional chain rules, it is natural to consider extending their methods to INLS. We adopt the approach from \cite[Section 3]{KillipVisan13} to avoid Besov-type spaces.

\vspace{0.6em}
Before the statement, we start with noticing some facts about parameters.
Since $n = 3$, $0<b<3/2$ always satisfies $\alpha > 1$, we may assume $n \geq 4$.
We also choose some more values that depend on or are restricted with respect to $n$ and $b$.
Let $0 < s, s_0 < 1$ and $1 < p,q < \infty$ satisfy the following relations.
\begin{equation*}
0 < s < \dfrac{\alpha}{1+\alpha}, \quad
s_0 = 1-s, \quad
p = \dfrac{2+\alpha}{s}, \quad
q' = \frac{q}{q-1} = \dfrac{2+\alpha}{(1+\alpha)s}
\end{equation*}
Notice that $s / \alpha < s_0$ and $2 < p < \infty$, while $q$ is not restricted to lie in $[2, \infty)$.

\vspace{0.6em}
Next, given a regularity $0\leq \gamma \leq 1$, we define two function spaces
\begin{align*}
X^{\gamma}(I) = {L_t^p(I,\dot{W}_x^{\gamma; \frac{1}{\frac{1}{2} - \frac{1-\gamma}{n} - \frac{2}{np}}, 2}(\mathbb{R}^n))}, \quad
Y^{\gamma}(I) = {L_t^{q'}(I,\dot{W}_x^{\gamma; \frac{1}{\frac{1}{2} + \frac{1+\gamma}{n} - \frac{2}{nq'}}, 2}(\mathbb{R}^n))}.
\end{align*}
Roughly speaking, these two function spaces play similar roles to the Strichartz and dual Strichartz spaces, respectively, which we can see from the embeddings $\dot{S}^1(I) \hookrightarrow X^{\gamma}(I)$ ($0 \leq \gamma \leq 1$) and $Y^{1}(I) \hookrightarrow \dot{N}^1(I)$ ($q \geq 2$).
We also see the following unusual Strichartz estimate, which is reminiscent of \cite[Lemma 3.2]{TaoVisan05} and \cite[Lemma 3.10]{KillipVisan13}. (The idea of proof is similar too.)
\begin{equation} \label{strichartz-exotic}
\norm{X^s(I)} {\int_{t_0}^{t} e^{i(t-\tau)\Delta} F(\tau) \,d\tau}
\lesssim \norm{Y^s(I)} {F}
\end{equation}
We further notice the following analogs of \cite[Lemma 3.12]{KillipVisan13} by the weighted fractional chain rules.
\begin{align}
\label{nonl-est-exotic-1}
\norm{Y^s (I)}{|x|^{-b} |u|^\alpha u}
&\lesssim \norm{X^0 (I)}{u} ^ \alpha \norm{X^s (I)}{u} \\
\label{nonl-est-exotic-2}
\norm{Y^s (I)}{|x|^{-b} |u|^\alpha w}
&\lesssim \norm{X^0 (I)}{w} \norm{X^0 (I)}{u} ^ {\alpha - \frac{s}{s_0}} \norm{X^{s_0} (I)}{u} ^ {\frac{s}{s_0}}
+ \norm{X^s (I)}{w} \norm{X^0 (I)}{u} ^ {\alpha}
\end{align}
For both \eqref{nonl-est-exotic-1} and \eqref{nonl-est-exotic-2}, the fractional chain rules are applicable when
\begin{equation*}
\frac{b+s}{n} < \frac{1}{2} + \frac{1+s}{n} - \frac{2}{nq'},
\end{equation*}
which always holds for our choice of the new parameters $s, s_0, p, q$ since
\begin{equation*}
\Big( \frac{1}{2} + \frac{1+s}{n} - \frac{2}{nq'} \Big) - \frac{b+s}{n}
= \frac{(n + 2 - 2b)(n - b - 2s)}{2n(n-b)}
> 0.
\end{equation*}

\vspace{0.6em}
Now we proceed to develop the short-time perturbation theory for \eqref{INLS} with $\alpha \leq 1$.

\begin{lem}[Short-time perturbation for INLS, $\alpha \leq 1$]~ \label{shortperturb}

Assume that $n \geq 4$, $0 < b < 2$ and $0 < \alpha \leq 1$.
Let $I$ be a compact time interval, and
let $\tilde{u} : I \times \mathbb{R}^n \to \mathbb{C}$ be an approximate solution to \eqref{INLS} in the sense that
\begin{equation*}
i\tilde{u}_t + \Delta\tilde{u} = \mu |x|^{-b} |\tilde{u}|^\alpha \tilde{u} + e
\end{equation*}
for some function $e$ representing error.
Assume that
\begin{equation*}
\norm{L_t^\infty(I, \dot{H}_x^1)}{\tilde{u}} \leq E
\end{equation*}
for some constant $E>0$.
Let $t_0 \in I$ and assume also that $u(t_0) \in \dot{H}_x^1 (\mathbb{R}^n)$ obeys
\begin{equation}
\label{shortperturb-small-0}
\norm{\dot{H}_x^1}{u(t_0) - \tilde{u}(t_0)} \leq D
\end{equation}
for some constant $D>0$.
Lastly, assume also the smallness conditions
\begin{align}
\label{shortperturb-small-1}
\norm{X^{s}(I)} {\tilde{u}} &\leq \delta, \\
\label{shortperturb-small-2}
\norm{X^{s}(I)} {e^{i(t-t_0)\Delta} (u(t_0) - \tilde{u}(t_0))} &\leq \epsilon, \\
\label{shortperturb-small-3}
\norm{\dot{N}^1(I)}{e} &\leq \epsilon
\end{align}
for some small constants $0 < \delta = \delta(E)$ and $0 < \epsilon < \epsilon_0 = \epsilon_0(E,D)$.
Then, there exists a unique solution $u : I\times\mathbb{R}^n \to \mathbb{C}$ of \eqref{INLS} with the initial data $u(t_0)$ at time $t = t_0$ satisfying
\begin{align}
\label{shortperturb-res-1}
\norm{X^{s}(I)} {u-\tilde{u}} &\lesssim \epsilon, \\
\label{shortperturb-res-3}
\norm{\dot{S}^1(I)} {u - \tilde{u}} &\lesssim D, \\
\label{shortperturb-res-2}
\norm{\dot{S}^1(I)} {u} &\lesssim E+D, \\
\label{shortperturb-res-4}
\norm{Y^{s}(I)} {|x|^{-b}(|u|^\alpha u - |\tilde{u}|^\alpha \tilde{u})} &\lesssim \epsilon, \\
\label{shortperturb-res-5}
\norm{\dot{N}^1(I)} {|x|^{-b}(|u|^\alpha u - |\tilde{u}|^\alpha \tilde{u})} &\leq D.
\end{align}

\end{lem}

\begin{proof}
The overall idea follows Killip and Visan \cite[Lemma 3.13]{KillipVisan13}. The use of Lorentz spaces is inspired by Aloui and Tayachi \cite{AlouiTayachi}.

\vspace{0.6em}
\emph{Part 1.}
We derive some bounds on the magnitude of $u$ and $\tilde{u}$.

\vspace{0.6em}
First, by Strichartz estimates, Proposition \ref{fract-chain-inhomog}, the smallness conditions, and some additional interpolation of spaces, we have
\begin{align*}
\norm{\dot{S}^1(I)}{\tilde{u}}
&\lesssim \norm{L_t^\infty(I,\dot{H}_x^1)} {\tilde{u}}
+ \norm{L_t^{2}(I,\dot{W}_x^{1; \frac{2n}{n+2}, 2})} {|x|^{-b} |\tilde{u}|^\alpha \tilde{u}}
+ \norm{\dot{N}^1(I)}{e} \\
&\lesssim E
+ \norm{L_t^{\frac{2(n+2)}{n-2}}(I, L_x^{\frac{2(n+2)}{n-2}})} {\tilde{u}} ^ \alpha
\norm{L_t^{p_b}(I, \dot{W}_x^{1; {q_b}, 2})} {\tilde{u}}
+ \epsilon \\
&\lesssim E
+ \delta ^ {\alpha c}
\norm{\dot{S}^1(I)} {\tilde{u}} ^ {1 + \alpha (1-c)}
+ \epsilon
\end{align*}
which leads to
\begin{equation}
\label{shortperturb-utilde-S1-bdd}
\norm{\dot{S}^1(I)}{\tilde{u}} \lesssim E
\end{equation}
provided that $\delta = O(1)$ and $\epsilon_0 = O(E)$.

\vspace{0.6em}
Next, by \eqref{strichartz-exotic} and \eqref{nonl-est-exotic-1}, we have
\begin{align*}
\norm{X^{s}(I)} {e^{i(t-t_0)\Delta} \tilde{u}(t_0)}
&\lesssim \norm{X^{s}(I)} {\tilde{u}}
+ \norm{Y^s(I)} {|x|^{-b} |\tilde{u}|^\alpha \tilde{u}}
+ \norm{\dot{N}^1(I)}{e} \\
&\lesssim \delta
+ \norm{X^{0}(I)} {\tilde{u}} ^ \alpha
\norm{X^{s}(I)} {\tilde{u}}
+ \epsilon \\
&\lesssim \delta
+ \delta ^ {1+\alpha}
+ \epsilon
\end{align*}
which leads to
\begin{equation}
\label{shortperturb-p1-eq2}
\norm{X^{s}(I)} {e^{i(t-t_0)\Delta} \tilde{u}(t_0)} \lesssim \delta
\end{equation}
provided that $\delta = O(1)$ and $\epsilon_0 = O(\delta)$.

\vspace{0.6em}
Lastly, by \eqref{shortperturb-small-2} and \eqref{shortperturb-p1-eq2} in addition, we have
\begin{align*}
\norm{X^s(I)}{u}
&\lesssim \norm{X^s(I)} {e^{i(t-t_0)\Delta}u(t_0)}
+ \norm{Y^s(I)} {|x|^{-b} |u|^\alpha u} \\
&\lesssim \delta + \epsilon
+ \norm{X^{0}(I)} {u} ^ \alpha
\norm{X^{s}(I)} {u}
\end{align*}
and thus
\begin{equation}
\label{shortperturb-u-Xs-small}
\norm{X^s(I)}{u} \lesssim \delta
\end{equation}
under the same conditions on $\delta$ and $\epsilon_0$.

\vspace{0.6em}
\emph{Part 2.}
We derive the desired bounds on the magnitude of the solution's approximation error. Let $w = u - \tilde{u}$.
Notice that $w$ is a solution to the equation
\begin{equation*}
\left\{
\begin{aligned}
&iw_t + \Delta w = \mu |x|^{-b} (|u|^\alpha u - |\tilde{u}|^\alpha \tilde{u}) - e, \\
&w(t_0) = u(t_0) - \tilde{u}(t_0).
\end{aligned}
\right.
\end{equation*}
From the equation of $w$, \eqref{nonl-est-exotic-2}, \eqref{shortperturb-small-1}--\eqref{shortperturb-small-3}, and \eqref{shortperturb-utilde-S1-bdd}, we calculate
\begin{align*}
\norm{X^s(I)}{w}
&\lesssim \norm{X^s(I)} {e^{i(t-t_0)\Delta} w(t_0)}
+ \norm{Y^s(I)} {|x|^{-b} (|u|^\alpha u - |\tilde{u}|^\alpha \tilde{u})}
+ \norm{\dot{N}^1(I)} {e} \\
&\lesssim \epsilon
+ \norm{Y^s(I)} {|x|^{-b} (|u|^\alpha u - |\tilde{u}|^\alpha \tilde{u})}
+ \epsilon \\
&\lesssim 2\epsilon
+ \norm{X^s(I)} {w}
\left( \norm{X^0(I)} {\tilde{u}} ^ {\alpha - \frac{s}{s_0}} \norm{X^{s_0}(I)} {\tilde{u}} ^ {\frac{s}{s_0}} +
\norm{X^0(I)} {w} ^ {\alpha - \frac{s}{s_0}} \norm{X^{s_0}(I)} {w} ^ {\frac{s}{s_0}} \right) \\
&\lesssim 2\epsilon
+ \norm{X^s(I)} {w}
\left( \delta ^ {\alpha - \frac{s}{s_0}} E ^ {\frac{s}{s_0}} +
\norm{X^s(I)} {w} ^ {\alpha - \frac{s}{s_0}} \norm{\dot{S}^{1}(I)} {w} ^ {\frac{s}{s_0}} \right)
\end{align*}
and hence
\begin{equation}
\label{shortperturb-w-Xs-est}
\norm{X^s(I)}{w}
\lesssim \epsilon
+ \norm{X^s(I)} {w} ^ {1 + \alpha - \frac{s}{s_0}} \norm{\dot{S}^{1}(I)} {w} ^ {\frac{s}{s_0}}
\end{equation}
provided that $\delta = O(E^{-\frac{s}{\alpha s_0 - s}})$.

\vspace{0.6em}
With \eqref{shortperturb-small-0} in place of \eqref{shortperturb-small-2}, and considering
\eqref{shortperturb-utilde-S1-bdd} and
\eqref{shortperturb-u-Xs-small}, a similar argument indicates that
\begin{align*}
\norm{\dot{S}^1(I)}{w}
&\lesssim \norm{\dot{S}^1(I)} {e^{i(t-t_0)\Delta} w(t_0)}
+ \norm{\dot{N}^1(I)} {|x|^{-b} (|u|^\alpha u - |\tilde{u}|^\alpha \tilde{u})}
+ \norm{\dot{N}^1(I)} {e} \\
&\lesssim D
+ \norm{L_t^2(I, \dot{W}_x^{1; \frac{2n}{n+2}, 2})} {|x|^{-b} (|u|^\alpha u - |\tilde{u}|^\alpha \tilde{u})}
+ \epsilon \\
&\lesssim D + \epsilon
+ \norm{\dot{S}^1(I)} {\tilde{u}} \norm{X^0(I)} {w} ^ {\alpha}
+ \norm{\dot{S}^1(I)} {w} \norm{X^0(I)} {u} ^ {\alpha} \\
&\lesssim D + \epsilon
+ E \norm{X^0(I)} {w} ^ {\alpha}
+ \delta ^ {\alpha} \norm{\dot{S}^1(I)} {w}
\end{align*}
and therefore
\begin{equation}
\label{shortperturb-w-S1-est}
\norm{\dot{S}^1(I)}{w}
\lesssim D
+ E \norm{X^s(I)} {w} ^ {\alpha}.
\end{equation}

\vspace{0.6em}
Combining \eqref{shortperturb-w-Xs-est} and \eqref{shortperturb-w-S1-est}, we prove the two properties \eqref{shortperturb-res-1} and \eqref{shortperturb-res-3}:
\begin{equation*}
\norm{X^s(I)}{w} \lesssim \epsilon \text{\quad and\quad}
\norm{\dot{S}^1(I)}{w} \lesssim D,
\end{equation*}
provided that $\epsilon_0 > 0$ is small enough (with polynomial decay with respect to large $E$ and $D$).
For the remaining three properties,
\eqref{shortperturb-res-2} follows by adding \eqref{shortperturb-utilde-S1-bdd} and \eqref{shortperturb-res-3},
\eqref{shortperturb-res-4} follows from the proof of \eqref{shortperturb-w-Xs-est},
and \eqref{shortperturb-res-5} follows from the proof of \eqref{shortperturb-w-S1-est}.
\end{proof}

\vspace{0.6em}
We finally proceed to state the long-time perturbation theory (i.e., stability theory) for {INLS} with $\alpha \leq 1$.

\begin{lem}[Stability for INLS, $\alpha \leq 1$]~ \label{longperturb}

Assume that $n \geq 4$, $0 < b < 2$ and $0 < \alpha \leq 1$.
Let $I$ be a compact time interval, and
let $\tilde{u} : I \times \mathbb{R}^n \to \mathbb{C}$ be an approximate solution to \eqref{INLS} in the sense that
\begin{equation*}
i\tilde{u}_t + \Delta\tilde{u} = \mu |x|^{-b} |\tilde{u}|^\alpha \tilde{u} + e
\end{equation*}
for some function $e$ representing error.
Assume that
\begin{equation*}
\norm{L_t^\infty(I, \dot{H}_x^1)}{\tilde{u}} \leq E, \quad
\norm{L_{t}^{\frac{2(n+2)}{n-2}}(I,L_{x}^{\frac{2(n+2)}{n-2}})}{\tilde{u}} \leq L
\end{equation*}
for some constants $E, L > 0$.
Let $t_0 \in I$ and assume also that $u(t_0) \in \dot{H}_x^1 (\mathbb{R}^n)$ obeys
\begin{equation}
\label{longperturb-small-1}
\norm{\dot{H}_x^1}{u(t_0) - \tilde{u}(t_0)} \leq D
\end{equation}
for some constant $D>0$.
Lastly, assume also the smallness conditions
\begin{align}
\label{longperturb-small-2.1}
\norm{L_{t}^{\frac{2(n+2)}{n-2}}(I,L_{x}^{\frac{2(n+2)}{n-2},2})} {e^{i(t-t_0)\Delta}(u(t_0) - \tilde{u}(t_0))} &\leq \epsilon, \\
\label{longperturb-small-2.2}
\norm{\dot{N}^1(I)}{e} &\leq \epsilon
\end{align}
for some small constant $0 < \epsilon < \epsilon_1 = \epsilon_1(E,L,D)$.
Then, there exists a unique solution $u : I\times\mathbb{R}^n \to \mathbb{C}$ of \eqref{INLS} with the initial data $u(t_0)$ at time $t = t_0$ satisfying
\begin{align}
\label{longperturb-res-1}
\norm{L_{t}^{\frac{2(n+2)}{n-2}}(I,L_{x}^{\frac{2(n+2)}{n-2}})} {u-\tilde{u}} &\leq C(E,L,D)\epsilon^{\theta}, \\
\label{longperturb-res-2}
\norm{\dot{S}^{1}(I)}{u-\tilde{u}} &\leq C(E,L,D)D, \\
\label{longperturb-res-3}
\norm{\dot{S}^{1}(I)}{u}&\leq C(E,L,D)
\end{align}
for a small absolute constant $\theta=\theta(n,b)\in(0,1)$.

\end{lem}

\begin{proof}
Again, the overall idea comes from Killip and Visan \cite[Proof of Theorem 3.8]{KillipVisan13}, with the same additional inspiration from Aloui and Tayachi \cite{AlouiTayachi} for Lorentz spaces.

\vspace{0.6em}
\emph{Part 1.}
We first show the boundedness of $\tilde{u}$ in $\dot{S}^1(\mathbb{R}^n)$.

\vspace{0.6em}
Let $0 < \eta \ll L$ be a small constant chosen later. We partition $I$ into $J_0 = J_0(L, \eta)$ subintervals $I_j = [t_j, t_{j+1}]$ ($j = 0, 1, \cdots, J_0-1$) such that
\begin{equation*}
\sup_j \norm{L_{t}^{\frac{2(n+2)}{n-2}}(I_j, L_{x}^{\frac{2(n+2)}{n-2}})} {\tilde{u}} \leq \eta.
\end{equation*}
Then by Strichartz estimates, Proposition \ref{fract-chain-inhomog}, and \eqref{longperturb-small-2.2}, we have
\begin{align*}
\norm{\dot{S}^1(I_j)}{\tilde{u}}
&\lesssim \norm{L_t^\infty(I_j,\dot{H}_x^1)} {\tilde{u}}
+ \norm{L_t^{2}(I_j,\dot{W}_x^{1; \frac{2n}{n+2}, 2})} {|x|^{-b} |\tilde{u}|^\alpha \tilde{u}}
+ \norm{\dot{N}^1(I_j)}{e} \\
&\lesssim E
+ \norm{L_t^{\frac{2(n+2)}{n-2}}(I_j, L_x^{\frac{2(n+2)}{n-2}})} {\tilde{u}} ^ \alpha
\norm{L_t^{p_b}(I_j, \dot{W}_x^{1; {q_b}, 2})} {\tilde{u}}
+ \epsilon \\
&\lesssim E
+ \epsilon
+ \eta ^ {\alpha}
\norm{\dot{S}^1(I)} {\tilde{u}}
\end{align*}
and therefore
$
\norm{\dot{S}^1(I_j)}{\tilde{u}} \lesssim E
$
for every $j$. Since $J_0 \sim (L/\eta)^c$, the summation in $j$ leads to
\begin{equation}
\label{longperturb-pf-utilde-S1}
\norm{\dot{S}^1(I)}{\tilde{u}} \lesssim E(L/\eta)^{c}
\end{equation}
provided that $\eta = O(1)$ and $\epsilon_1 = O(E)$.
In particular, by Sobolev embedding,
\begin{equation*}
\norm{X^s(I)}{\tilde{u}} \lesssim \norm{\dot{S}^1(I)}{\tilde{u}} \lesssim EL^c.
\end{equation*}

\vspace{0.6em}
\emph{Part 2.}
We would like to make some preparations before iterating the short-time perturbation theory, Lemma \ref{shortperturb}.
The hypotheses $\norm{L_t^\infty(I, \dot{H}_x^1)}{\tilde{u}} \leq E$ and $\norm{\dot{N}^1(I)}{e} \leq \epsilon$ are shared by both Lemma \ref{shortperturb} and this theorem.
For the smallness condition \eqref{shortperturb-small-1}, we let $\delta = \delta(E) > 0$ as defined in Lemma \ref{shortperturb} and then partition $I$ into $J_1 = J_1(E, L)$ subintervals $I_j = [t_j, t_{j+1}]$ ($j = 0,1,\cdots, J_1-1$) so that
\begin{equation*}
\sup_j \norm{X^s(I_j)} {\tilde{u}} \leq \delta.
\end{equation*}

Lastly, we see the following smallness condition, which is reminiscent of \eqref{shortperturb-small-2} and derived from \eqref{longperturb-small-1}, \eqref{longperturb-small-2.1}, Strichartz estimates, and interpolation.
This condition is used later in Duhamel's formula for $e^{i(t-t_j)\Delta} \left( u(t_j) - \tilde{u}(t_j) \right)$.
\begin{equation}
\begin{aligned}
\label{longperturb-Xs-interpol}
&\:
\norm{X^s(I)} {e^{i(t-t_0)\Delta}(u(t_0) - \tilde{u}(t_0))} \\
\lesssim&\:
\norm{L_{t}^{\frac{2(n+2)}{n-2}}(I,L_{x}^{\frac{2(n+2)}{n-2},2})} {e^{i(t-t_0)\Delta} (u(t_0) - \tilde{u}(t_0))} ^ {\theta}
\norm{\dot{S}^1(I)} {e^{i(t-t_0)\Delta} (u(t_0) - \tilde{u}(t_0))} ^ {1 - \theta} \\
\lesssim&\:
\epsilon^{\theta} D^{1-\theta}
\lesssim
\epsilon^{\theta/2}
\end{aligned}
\end{equation}
Since $2<p<\infty$, by using Sobolev embedding when necessary, it is always possible to find the appropriate  constant $0<\theta<1$ under the assumption $\epsilon_1 = O(D^{-2(1-\theta)/\theta})$.

\vspace{0.6em}
\emph{Part 3.}
Let $w = u - \tilde u$.
Choose $\epsilon_1 = \epsilon_1(E,L,D)$ so that $\epsilon_1^{\theta/2} \ll \epsilon_0(E, C(J_1)D)$.
Then, we may apply Lemma \ref{shortperturb} for every $j = 0, \cdots, J_1-1$ to obtain
\begin{equation}
\label{longperturb-iterresult}
\begin{aligned}
\norm{X^{s}(I_j)} {w} &\lesssim C(j) \epsilon^{\theta/2}, \\
\norm{\dot{S}^1(I_j)} {w} &\lesssim C(j) D, \\
\norm{\dot{S}^1(I_j)} {u} &\lesssim C(j) (E+D), \\
\norm{Y^{s}(I_j)} {|x|^{-b}(|u|^\alpha u - |\tilde{u}|^\alpha \tilde{u})} &\lesssim C(j) \epsilon^{\theta/2}, \\
\norm{\dot{N}^1(I_j)} {|x|^{-b}(|u|^\alpha u - |\tilde{u}|^\alpha \tilde{u})} &\leq C(j) D
\end{aligned}
\end{equation}
once we show that
\begin{equation}
\label{longperturb-keyscheme}
\norm{\dot{H}_x^1}{w(t_j)} \leq C(j)D, \quad
\norm{X^{s}(I_j)} {e^{i(t-t_j)\Delta}w(t_j)} \leq C(j) \epsilon^{\theta/2}
\end{equation}
for every $j$.
To Duhamel's formula for $e^{i(t-t_j)\Delta} w(t_j)$, which is
\begin{equation*}
e^{i(t-t_j)\Delta}w(t_j) =
e^{i(t-t_0)\Delta}w(t_0) - i\int_{t_0}^{t_j} e^{i(t-\tau)\Delta} \big( \mu|x|^{-b}(|u|^\alpha u - |\tilde u|^\alpha \tilde u) - e \big)(\tau) \,d\tau,
\end{equation*}
we apply \eqref{strichartz-exotic} to derive the estimate
\begin{align*}
\norm{X^s(I_j)} {e^{i(t-t_j)\Delta}w(t_j)}
&\lesssim
\norm{X^s(I_j)} {e^{i(t-t_0)\Delta}w(t_0)} +
\norm{Y^s([t_0, t_j])} {|x|^{-b}(|u|^\alpha u - |\tilde u|^\alpha \tilde u)} +
\norm{\dot{N}^1(I_j)} {e}
\\ &\lesssim
\epsilon^{\theta/2} + \sum_{k=0}^{j-1} C(k) \epsilon^{\theta/2} + \epsilon,
\end{align*}
and similarly, we also obtain
\[
\norm{\dot{H}_x^1} {w(t_j)}
\lesssim
D + \sum_{k=0}^{j-1} C(k) D + \epsilon.
\]
Therefore, taking $\displaystyle C(j) \sim 1 + \sum_{k=0}^{j-1} C(k)$ and assuming $\epsilon_1 = O(\langle D\rangle^{-1})$ in addition, we obtain \eqref{longperturb-keyscheme}. Summation of \eqref{longperturb-iterresult} over $j$ and the interpolation in the reverse direction of \eqref{longperturb-Xs-interpol}, we obtain \eqref{longperturb-res-1}--\eqref{longperturb-res-3}.

\end{proof}

\subsection{Variational analysis} \label{Subsection:2.3}

Now that the stability theory has been established for every $n \geq 3$ and $0<b<\min(2,n/2)$,
we finish the section with the variational analysis.
The argument is highly similar to \cite{GuzmanXu24.6}, and the only difference is that it holds for every such $n$ and $b$. Thus, we only provide statements without proof.

\vspace{0.6em}
It is well known by Yanagida \cite[Remark 2.1]{Yanagida} that
the ground state
\[
Q(x) = \left(1 + \frac{|x|^{2-b}}{(n-b)(n-2)} \right)^{-\frac{n-2}{2-b}}
\]
is the positive, smooth, and radially decreasing solution of the elliptic equation $ \Delta Q + |x|^{-b}Q^{\alpha + 1} = 0 $ which is unique up to scaling.
In addition, $Q$ attains the optimal constant $C_1$ for the embedding estimate
\[
\int_{\mathbb{R}^n} |x|^{-b} |u|^{\alpha+2} \,dx \leq C_1 \left( \int_{\mathbb{R}^n} |\nabla u|^{2} \,dx \right)^{\frac{\alpha+2}{2}}
\]
and also satisfies
\[
C_1^{-\frac{2}{\alpha}}
= \int_{\mathbb{R}^n} |\nabla Q|^{2} \,dx = \int_{\mathbb{R}^n} |x|^{-b} |Q|^{\alpha+2} \,dx
= \frac{2\pi^{\frac{n}{2}}((n-2)(n-b))^{\frac{n-b}{2-b}} \Gamma \Big(\dfrac{n-b}{2-b}\Big)^2} {(2-b)\Gamma\Big(\dfrac{n}{2}\Big)\Gamma \Big(\dfrac{2(n-b)}{2-b}\Big)}.
\]

\begin{lem}[Energy trapping and coercivity]~

Suppose that a solution $u : I \times \mathbb{R}^n \to \mathbb{C}$ of \eqref{INLS} obeys the subthreshold condition
\[
E[\phi] < (1-\delta_0) E[Q] \text{\quad and\quad} \norm{L_x^2}{\nabla \phi} < \norm{L_x^2}{\nabla Q}.
\]
Then there exists $\delta = \delta(n, b, \delta_0) > 0$ such that
\[
\sup_{t\in I} \norm{L_x^2}{\nabla u(t)} < (1-\delta) \norm{L_x^2}{\nabla Q}.
\]
Furthermore, we have the comparison
\[
E[u(t)] \sim \norm{L_x^2}{\nabla u(t)}^2 \sim \norm{L_x^2}{\nabla \phi}^2
\text{\quad uniformly in $t\in I$,}
\]
where the implicit constants depend only on $n$ and $b$.
Lastly, there exists $\delta' = \delta'(n, b, \delta_0) > 0$ such that
\[
\int_{\mathbb{R}^n} |\nabla u(t)|^2-|x|^{-b}|u(t)|^{\alpha+2} \,dx
\geq
\delta' \int_{\mathbb{R}^n} |\nabla u(t)|^2\,dx.
\]
\end{lem}

\newpage
\section{Construction of a minimal energy blow-up solution} \label{Section:3}

In this section, our main goal is to construct a \emph{minimal energy blow-up solution} (also called a \emph{critical solution}) with a compactness property in the sense of Proposition \ref{compactness-of-orbit}.
The road map mostly follows Guzman and Murphy \cite{GuzmanMurphy21} and its generalization Guzman and Xu \cite{GuzmanXu24.6}.
It is based on the concentration--compactness argument of Keng and Merle \cite{KenigMerle}
with an additional technique, \cite[Proposition 3.3]{GuzmanMurphy21}, which prevents the blow-up profiles from escaping from the origin too quickly.
However, some intermediate norm analyses need Lorentz spaces for controlling the harder cases, especially $\alpha < 1$.
To avoid duplication, we provide an overall sketch of the concentration compactness arguments most of the time, and detailed computations only for necessary places.


\vspace{0.6em}
We begin with the linear profile decomposition developed by Keraani \cite{Keraani01}.
See also Koch, Tataru, and Visan \cite[pp. 223--250]{KochTataruVisan14}.
For completeness, we write down the following proposition despite its similarity to \cite{GuzmanMurphy21} and \cite{GuzmanXu24.6}
After some direct computations, it is possible to verify that the proposition holds also for the values of $(n, b)$ which have not been covered in \cite{GuzmanXu24.6}.

\begin{prop}[Linear profile decomposition] \label{linear-prof-decomp}~

Let $n\geq 3$, $0<b<\min(2,n/2)$.
Let $u_m$ be a bounded sequence in $\dot{H}^1(\mathbb{R}^n)$.
Then there exist $J^* \in \mathbb{N} \cup \{\infty\}$,
profiles $0 \ne \phi^{j} \in \dot{H}^1(\mathbb{R}^n)$,
scaling parameters $0 < \lambda_m^j < \infty$,
space translation parameters $x_m^j \in \mathbb{R}^n$,
time translation parameters $t_m^j \in \mathbb{R}$,
and the remainder terms $w_m^J$ such that
\begin{equation*}
u_m = \sum_{j=1}^{J} g_m^j[e^{i t_m^j \Delta} \phi^j] + w_m^J
\end{equation*}
for $1\leq J \leq J^*$, where the energy-critical symmetry operator $g_m^j : \dot{H}^1(\mathbb{R}^n) \to \dot{H}^1(\mathbb{R}^n)$ is defined as
\begin{equation*}
(g_m^j f)(x) = (\lambda_m^j)^{-\frac{n-2}{2}} f\Big( \frac{x-x_m^j}{\lambda_m^j} \Big).
\end{equation*}
The decomposition additionally satisfies
\begin{itemize}
\item[(i)] Energy decoupling: for every $1\leq J \leq J^*$, we have
\begin{align*}
\lim_{m\to\infty} \Big( K(u_m) - \sum_{j=1}^{J} K(\phi^j) - K(w_m^J) \Big) &= 0, \\
\lim_{m\to\infty} \Big( P(u_m) - \sum_{j=1}^{J} P(\phi^j) - P(w_m^J) \Big) &= 0
\end{align*}
where we write the ``kinetic energy'' $K(f)$ and ``potential energy'' $P(f)$ as follows respectively.
\begin{equation*}
K(f) = \int_{\mathbb{R}^n} |\nabla f|^2 dx, \quad
-P(f) = \int_{\mathbb{R}^n} |x|^{-b}|f|^{\alpha + 2} dx
\end{equation*}

\item[(ii)]
Asymptotic (strong) vanishing of the remainder:
\begin{equation*}
\lim_{J\to J^*} \limsup_{m\to\infty} \norm{L_{t}^{\frac{2(n+2)}{n-2}}(\mathbb{R}, L_{x}^{\frac{2(n+2)}{n-2}})} {e^{it\Delta} w_m^J} = 0.
\end{equation*}

\item[(iii)]
(Asymptotic) Weak vanishing of the remainder: For every $1 \leq j \leq J < \infty$,
\[
e^{-i{t_m^j}\Delta} \big[ (g_m^j)^{-1} w_m^J \big]
\xrightharpoonup{\quad\!} 0
\text{\quad weakly in $\dot{H}_x^1(\mathbb{R}^n)$}.
\]

\item[(iv)]
Asymptotic orthogonality: for every $j \ne j'$, we have
\begin{equation*}
\lim_{m\to\infty} \Big( \frac{\lambda_m^j}{\lambda_m^{j'}} + \frac{\lambda_m^{j'}}{\lambda_m^j} + \frac{|x_m^j - x_m^{j'}|^2}{\lambda_m^j \lambda_m^{j'}} + \frac{|t_m^j(\lambda_m^j)^2 - t_m^{j'}(\lambda_m^{j'})^2|}{\lambda_m^j \lambda_m^{j'}} \Big) = \infty.
\end{equation*}
\end{itemize}
In addition, we may assume either $t_m^j \equiv 0$ or $t_m^j \to \pm\infty$ and that either $x_m^j \equiv 0$ or $|x_m^j| \to \infty$.

\end{prop}

\vspace{0.6em}
We next visit the technique \cite[Proposition 3.3]{GuzmanMurphy21} for Guzman and Murphy and \cite[Proposition 3.2]{GuzmanXu24.6} for Guzman and Xu.
They prevent the space-translational escape of blow-up profiles in Proposition \ref{palaissmale}, which has not occurred under the radial data assumption.
They also provide a smooth approximation which enables the asymptotic orthogonality.
The authors have utilized the vanishing of $|x|^{-b}$ at infinity to address the problem in INLS.
Here, we utilize Lorentz spaces to push their techniques further to hold for every $n\geq 3$ and $0<b<\min(2,n/2)$.
Aside from the overall sketch, only the significantly altered parts from the proof in \cite{GuzmanMurphy21} and \cite{GuzmanXu24.6} will be pointed out.

\vspace{0.6em}
Before writing down the statement,
given the spatial and scaling parameter sequences ($x_m \in \mathbb{R}^n$ and $0 < \lambda_m < \infty$),
we define the operator $g_m : \dot{H}^1(\mathbb{R}^n) \to \dot{H}^1(\mathbb{R}^n)$ by
\begin{equation*}
g_m f(x) = (\lambda_m)^{-\frac{n-2}{2}} f\Big( \frac{x-x_m}{\lambda_m} \Big)
\end{equation*}
which is reminiscent of $g_m^j$ in Proposition \ref{linear-prof-decomp}.
Note by direct calculation that
$
g_m e^{it\Delta} = e^{i\lambda_m^2 t\Delta} g_m
$.

\begin{prop} \label{scattering-step-2}~

Suppose that $x_m \in \mathbb{R}^n$, $0 < \lambda_m < \infty$, and $t_m \in \mathbb{R}$ satisfy
\begin{equation*}
\lim_{m\to\infty} \frac{|x_m|}{\lambda_m} = \infty, \text{\quad and either  }
t_m \equiv 0 \text{  or \ } t_m \to \pm\infty.
\end{equation*}
Let $\phi \in \dot{H}^1(\mathbb{R}^n)$ and define
\begin{equation*}
\phi_m(x) = g_m(e^{it_m\Delta}\phi)(x) = \lambda_m^{-\frac{n-2}{2}} (e^{it_m\Delta}\phi) \Big(\frac{x - x_m}{\lambda_m}\Big)
\end{equation*}
Then, for $m$ large enough, there exists a global solution $v_m$ to \eqref{INLS} with the initial data $v_m(0) = \phi_m$ satisfying
\begin{equation}
\label{prevent-disloc-by-inhomog-res-1}
\norm{\dot{S}^1(\mathbb{R})} {v_m} \lesssim 1
\end{equation}
where the implicit constant depends only on $E_c$ and $\norm{\dot{H}_x^1}{\phi}$.

In addition, let $(p,q)$ be any $L^2$-admissible pair with $p < \infty$.
Then for every $\epsilon > 0$, there exists $m^* \in \mathbb{N}$ and $\psi \in C_0^\infty(\mathbb{R} \times \mathbb{R}^n)$ such that for every $m \geq m^*$,
\begin{equation}
\label{prevent-disloc-by-inhomog-res-2}
\norm{L_{t}^{p}(\mathbb{R}, \dot{W}_{x}^{1; q, 2})} {\lambda_m^{\frac{n-2}{2}} v_m\big( \lambda_m^2(t-t_m), \lambda_m x + x_m \big) - \psi} < \epsilon.
\end{equation}
\end{prop}

\begin{proof}
The overall road map is similar to the proof of \cite[Proposition 3.2]{GuzmanXu24.6}.

\vspace{0.6em}
Let $0 < \theta < 1$ be a small constant to choose later.
We define a smooth frequency cutoff as the Littlewood-Paley operator
\begin{equation*}
P_m = P_{\left|\frac{x_m}{\lambda_m}\right|^{-\theta} \leq \cdot \leq \left|\frac{x_m}{\lambda_m}\right|^{\theta}},
\end{equation*}
and also define a smooth physical cutoff $\chi_m: \mathbb{R}^n \to [0, 1]$ which satisfies
\begin{equation*}
\chi_m(x) =
\begin{cases}
1 & \text{if } \Big|x + \dfrac{x_m}{\lambda_m}\Big| \geq \dfrac{1}{2} \Big|\dfrac{x_m}{\lambda_m}\Big|, \\[1ex]
0 & \text{if } \Big|x + \dfrac{x_m}{\lambda_m}\Big| \leq \dfrac{1}{4} \Big|\dfrac{x_m}{\lambda_m}\Big|
\end{cases}
\end{equation*}
and also for every multiindex $\iota \in (\mathbb{N} \cup \{0\})^n$ the bound condition
\begin{equation*}
\norm{L_x^\infty}{\partial^{\iota} \chi_m} \leq C(\iota) \Big|\frac{x_m}{\lambda_m}\Big|^{-|\iota|}.
\end{equation*}
From the hypothesis, we note that $P_m \to {\rm id}$ as an operator and $\chi_m \to 1$ pointwise as $m \to \infty$.

\vspace{0.6em}
Let $a_{m,T}^- = (-t_m - T)\lambda_m^2$ and $a_{m,T}^+ = (-t_m + T)\lambda_m^2$.
We partition the whole time interval $\mathbb{R}$ into three pieces;
$I_{m,T}^\circ = [a_{m,T}^-, a_{m,T}^+]$,
$I_{m,T}^- = (-\infty, a_{m,T}^-]$,
and $I_{m,T}^+ = [a_{m,T}^+, \infty)$.
We then define an approximate solution $\tilde{v}_{m,T}$ by
\begin{equation*}
\tilde{v}_{m,T}(t,x) =
\begin{cases}
g_m \big[ \chi_m P_m e^{i(t \lambda_m^{-2} + t_m)\Delta} \phi \big](x)
= \chi_m\Big( \dfrac{x-x_m}{\lambda_m} \Big) e^{i(t + t_m \lambda_m^2)\Delta} g_m [P_m \phi](x)
& \text{if } t \in I_{m,T}^\circ, \\[1ex]
e^{i(t-a_{m,T}^+)\Delta} \tilde{v}_{m,T}(a_{m,T}^+)
& \text{if } t \in I_{m,T}^+, \\[1ex]
e^{i(t-a_{m,T}^-)\Delta} \tilde{v}_{m,T}(a_{m,T}^-)
& \text{if } t \in I_{m,T}^-.
\end{cases}
\end{equation*}
We note that $\tilde{v}_{m,T}$ serves as an approximant to the true solution $v_m$ of \eqref{INLS} with the initial data $\tilde{v}_{m,T}(0) = \phi_m = g_m e^{it_m\Delta} \phi$.

\vspace{0.6em}
To prove \eqref{prevent-disloc-by-inhomog-res-1}, it suffices to check the following three conditions to apply the stability theory;
\begin{gather}
\label{scattering-nodisloc-approxsol-1}
\limsup_{T\to\infty} \limsup_{m\to\infty}
\norm{\dot{S}^1(\mathbb{R})} {\tilde{v}_{m,T}} \lesssim 1, \\
\label{scattering-nodisloc-approxsol-2}
\limsup_{T\to\infty} \limsup_{m\to\infty}
\norm{\dot{H}_x^1} {\tilde{v}_{m,T}(0) - \phi_m} = 0, \\
\label{scattering-nodisloc-approxsol-3}
\limsup_{T\to\infty} \limsup_{m\to\infty}
\norm{\dot{N}^1(\mathbb{R})} {e_{m,T}} = 0
\end{gather}
where
$ e_{m,T} = (i\partial_t + \Delta)\tilde{v}_{m,T} + |x|^{-b} |\tilde{v}_{m,T}|^\alpha \tilde{v}_{m,T} $
is the approximation error.
We note that the conditions \eqref{scattering-nodisloc-approxsol-1}, \eqref{scattering-nodisloc-approxsol-2}, and \eqref{scattering-nodisloc-approxsol-3} correspond to Conditions 1, 2, and 3 in the proof of \cite[Proposition 3.2]{GuzmanXu24.6}.
Since \eqref{scattering-nodisloc-approxsol-1} and
\eqref{scattering-nodisloc-approxsol-2} can be proved without much modification from the arguments in \cite{GuzmanXu24.6}, we focus on \eqref{scattering-nodisloc-approxsol-3}.

\vspace{0.6em}
We decompose $e_{m,T}$ as the sum of the linear part $e_{m,T}^{\rm l} = (i\partial_t + \Delta)\tilde{v}_{m,T}$ and the nonlinear part $e_{m,T}^{\rm nl} = |x|^{-b} |\tilde{v}_{m,T}|^\alpha \tilde{v}_{m,T}$.
We perform the estimation on each of the three subintervals $I_{m,T}^\circ$, $I_{m,T}^+$, and $I_{m,T}^-$, but whichever interval we consider, the control of $e_{m,T}^{\rm l}$ is relatively straightforward by the arguments in \cite{GuzmanXu24.6}. (In particular, $e_{m,T}^{\rm l} \equiv 0$ on $I_{m,T}^\pm$.)
Thus, we only consider $e_{m,T}^{\rm nl}$.

\vspace{0.6em}
Consider the time interval $I_{m,T}^\circ$.
Then, we can write
\[
e_{m,T}^{\rm nl} =
\lambda_m^{-(2-b)} g_m\left[|\lambda_m x + x_m|^{-b} \chi_m^{\alpha+1} |\Phi_m|^\alpha \Phi_m \right]
\]
where $\Phi_m = P_m e^{i(t \lambda_m^{-2} + t_m)\Delta} \phi$. This yields
\begin{align*}
&\; \norm{L_t^2(I_{m,T}^\circ; L_x^{\frac{2n}{n+2}, 2})} {\nabla e_{m,T}^{\rm nl}} \\
\lesssim&\;
\lambda_m^b T^{\frac{1}{2}}
\norm{L_t^\infty(I_{m,T}^\circ; L_x^{\frac{2n}{n+2}, 2})} {\nabla \big(|\lambda_m x + x_m|^{-b} \chi_m^{\alpha+1} |\Phi_m|^\alpha \Phi_m \big)} \\
\lesssim&\;
\lambda_m^b T^{\frac{1}{2}}
\norm{L_t^\infty(I_{m,T}^\circ; L_x^{\frac{1}{\frac{b+1}{n} - \epsilon}, \infty})} {\nabla \big(|\lambda_m x + x_m|^{-b} \chi_m^{\alpha+1} \big)}
\norm{L_t^\infty(I_{m,T}^\circ; L_x^{\frac{1}{\frac{n+2-2b}{2n} + \epsilon}, 2})} {\nabla \big(|\Phi_m|^\alpha \Phi_m \big)} \\
\lesssim&\;
\lambda_m^b T^{\frac{1}{2}}
\cdot \lambda_m^{-b+\epsilon n} |x_m|^{-\epsilon n} \cdot
\norm{L_t^\infty(I_{m,T}^\circ; L_x^{\frac{2(n+2)}{n-2}})} {\Phi_m}^\alpha
\norm{L_t^\infty(I_{m,T}^\circ; \dot{W}_x^{1; \frac{1}{\frac{n^2+4-4b}{2n(n+2)} + \epsilon}, 2})} {\Phi_m} \\
\lesssim&\;
T^{\frac{1}{2}}
\Big|\frac{x_m}{\lambda_m}\Big|^{-\epsilon n + \theta \left( \frac{n-2}{n+2}\alpha + \frac{n-2+2b}{n+2} - \epsilon n \right)}
\norm{\dot{H}_x^1} {\phi}^{\alpha+1}
\end{align*}
where the constant $\epsilon = \epsilon(n,b) > 0$ can be fixed small. We then choose $\theta = \theta(n, b)$ so that
\[
{-\epsilon n + \theta \left( \frac{n-2}{n+2}\alpha + \frac{n-2+2b}{n+2} - \epsilon n \right)} < 0.
\]
Therefore by the hypothesis $|x_m|/\lambda_m \to \infty$, we obtain
\[
\limsup_{T\to\infty} \limsup_{m\to\infty}
\norm{L_t^2(I_{m,T}^\circ; L_x^{\frac{2n}{n+2}, 2})} {\nabla e_{m,T}^{\rm nl}} = 0.
\]

On the interval $I_{m,T}^+$ ($I_{m,T}^-$ can be handled similarly.), we simply write
\[
e_{m,T} =
e_{m,T}^{\rm nl} =
|x|^{-b} |V_{m,T}|^\alpha V_{m,T}
\]
where $V_{m,T} = e^{i(t - a_{m,T}^+)\Delta} \tilde{v}_{m,T}(a_{m,T}^+)$.
By H\"older's inequality, Strichartz estimates, and \eqref{scattering-nodisloc-approxsol-1},
we have
\begin{align*}
&\; \norm{L_t^2(I_{m,T}^+; L_x^{\frac{2n}{n+2}, 2})} {\nabla e_{m,T}^{\rm nl}} \\
\lesssim&\;
\norm{L_t^{\frac{2(n+2)}{n-2}} (I_{m,T}^+; L_x^{\frac{2(n+2)}{n-2}})} {V_{m,T}}^\alpha
\norm{L_t^{\frac{2(n+2)}{n-2+2b}}(I_{m,T}^+; \dot{W}_x^{1; \frac{1}{\frac{n^2+4-4b}{2n(n+2)}}, 2})} {V_{m,T}} \\
\lesssim&\;
C(\norm{\dot{H}_x^1} {\phi}, E_c)
\norm{L_t^{\frac{2(n+2)}{n-2}} ([0,\infty); L_x^{\frac{2(n+2)}{n-2}})} {e^{it\Delta} \tilde{v}_{m,T}(a_{m,T}^+)}^\alpha.
\end{align*}
We observe that ${e^{it\Delta} \tilde{v}_{m,T}(a_{m,T}^+)} \to 0$ in ${L_t^{\frac{2(n+2)}{n-2}} ([0,\infty); L_x^{\frac{2(n+2)}{n-2}})}$ by dominated convergence as in \cite{GuzmanXu24.6},
finishing the proof of \eqref{prevent-disloc-by-inhomog-res-1}.

\vspace{0.6em}
We omit the proof of \eqref{prevent-disloc-by-inhomog-res-2} as we can borrow the idea from \cite{GuzmanMurphy21}.
Unlike \cite{GuzmanXu24.6}, we need a better regularity condition on \eqref{prevent-disloc-by-inhomog-res-2} to make use of the asymptotic orthogonality condition for $\alpha \leq 1$.

\end{proof}


We are now ready to explore the critical solution's compactness properties, which starts with establishing the Palais-Smale Condition.

\begin{prop} [Palais-Smale Condition] \label{palaissmale} ~

Let $u_m : I_m \times \mathbb{R}^n \to \mathbb{C}$ denote a sequence of maximal-lifespan solutions to \eqref{INLS} such that
\begin{align}
\label{palaissmale-H1-bddness}
& \lim_{m\to\infty} \norm{L_t^\infty (I_m, L_x^2)} {\nabla u_m}^2 = E_c < \norm{L_x^2}{\nabla W} ^2, \\
& \lim_{m\to\infty} \norm{L_t^{\frac{2(n+2)}{n-2}}([t_m,+\infty) \cap I_m; L_x^{\frac{2(n+2)}{n-2}} )} {u_m}
= \lim_{m\to\infty} \norm{L_t^{\frac{2(n+2)}{n-2}}((-\infty,t_m] \cap I_m; L_x^{\frac{2(n+2)}{n-2}})} {u_m}
= \infty.
\end{align}
Then there exist $\lambda_m \in (0,\infty)$ such that $\left\{ \lambda_m^{-\frac{n-2}{2}} u_m \Big(t_m, \dfrac{x}{\lambda_m}\Big) \right\}$ is precompact in $\dot{H}_x^1(\mathbb{R}^n)$.

\end{prop}

\begin{proof}
The overall road map is similar to the proof of \cite[Proposition 3.3]{GuzmanXu24.6}, but the existence of a bad profile is borrowed from \cite{KillipVisan10} as the higher-dimensional analogy.

\vspace{0.6em}
Assume $t_m \equiv 0$ by time translation, then
\[
\lim_{m\to\infty} \norm{L_t^{\frac{2(n+2)}{n-2}} ([0,\infty) \cap I_m, L_x^{\frac{2(n+2)}{n-2}} )} {u_m}
= \lim_{m\to\infty} \norm{L_t^{\frac{2(n+2)}{n-2}} ((-\infty,0]\cap I_m, L_x^{\frac{2(n+2)}{n-2}} )} {u_m}
= \infty.
\]
Since $\displaystyle \sup_m \norm{\dot{H}_x^1} {u_m(0)}^2 < \infty$ by \eqref{palaissmale-H1-bddness}, by passing to a subsequence in $m$, we can apply the linear profile decomposition of Proposition \ref{linear-prof-decomp} to $u_m(0)$ by
\begin{equation*}
u_m(0) = \sum_{j=1}^{J} g_m^j[e^{i t_m^j \Delta} \phi^j] + w_m^J.
\end{equation*}

If $\displaystyle \limsup_{m\to\infty} \Big|\frac{x_m^j}{\lambda_m^j}\Big| = \infty$,
then we apply Proposition \ref{scattering-step-2}.
In this case, passing to a subsequence in $m$, we define $v_m^j$ as the global solution to \eqref{INLS} with the initial data
$v_m^j(0) = g_m^j(e^{it_m^j \Delta} \phi^j)$, which satisfies $\norm{L_t^{\frac{2(n+2)}{n-2}} (\mathbb{R}, L_x^{\frac{2(n+2)}{n-2}} )} {v_m^j} < \infty$.
On the other hand,
if $\displaystyle \limsup_{m\to\infty} \Big|\frac{x_m^j}{\lambda_m^j}\Big| < \infty$,
then we may assume $x_m^j \equiv 0$ (passing to a subsequence in $m$ again and translating $\phi^j$ when necessary)
and construct the nonlinear profiles $v^j$ similar to \cite[Definition 2.12]{KenigMerle}, associated to $(\phi^j, \{t_m^j\})$ and satisfying
$
\lim\limits_{m\to\infty} \norm{\dot{H}_x^1} {v^j(t_m^j) - e^{it_m^j \Delta} \phi^j} = 0.
$
In this case, we define $\displaystyle v_m^j(t,x) = \lambda_m^{-\frac{n-2}{2}} v^j\Big(\frac{t}{(\lambda_m^j)^2}+t_m^j, \frac{x}{\lambda_m^j}\Big)$, whose initial data becomes $v_m^j(0) = g_m^j(v^j(t_m^j))$.
With the profiles $v_m^j$ settled, we define the approximate solution $u_m^J$ of $u_m$ by
\begin{equation}
\label{def:approxsol-umJ}
u_m^J = \sum_{j=1}^{J} v_m^j + e^{it\Delta} w_m^J.
\end{equation}

\vspace{0.6em}
Due to the energy decoupling of the linear profile decomposition (See Proposition \ref{linear-prof-decomp}, (i)), there exist $J_0 \in \mathbb{N}$ such that for every $j \geq J_0$ (i.e. when $j$ is large), $\norm{\dot{H}_x^1} {\phi^j}$ is small enough so that the solution $v_m^j$ is globally well-posed and satisfies
\[
\sup_m \bigg( \norm{L_t^\infty(\mathbb{R}, \dot{H}_x^1)} {v_m^j}
+ \norm{L_t^{\frac{2(n+2)}{n-2}}(\mathbb{R}, L_x^{\frac{2(n+2)}{n-2}})} {v_m^j} \bigg)
\lesssim \norm{\dot{H}_x^1} {\phi^j}.
\]
Here, we propose the following claim.

\begin{clm} [At least one bad profile] \label{clm:badprofile} ~

There exists $1\leq j_0 < J_0$ such that
\[
\limsup_{m\to\infty} \norm{L_t^{\frac{2(n+2)}{n-2}} ([0,\infty) \cap I_m^{j_0}, L_x^{\frac{2(n+2)}{n-2}} )} {v_m^{j_0}} = \infty.
\]
\end{clm}

Suppose that $\displaystyle \sup_{1\leq j<J_0} \limsup_{m\to\infty} \norm{L_t^{\frac{2(n+2)}{n-2}} ([0,\infty), L_x^{\frac{2(n+2)}{n-2}} )} {v_m^j} < \infty$ (i.e. $\displaystyle \sup I_m^{j_0} = \infty$) on the contrary.
We aim to show the following four conditions of the approximate solution $u_m^J$;
\begin{align}
\label{pf-badprofile-o/w-1}
\sup_{1\leq j \leq J_0} \limsup_{m\to\infty} \norm{\dot{S}^1([0,\infty))} {v_m^j} &< \infty,
\\
\label{pf-badprofile-o/w-2}
\sup_{1\leq j \leq J^*} \limsup_{m\to\infty} \norm{\dot{H}_x^1} {u_m(0) - u_m^J(0)} &= 0,
\\
\label{pf-badprofile-o/w-3}
\limsup_{J \to J^*} \limsup_{m\to\infty} \norm{L_t^{\frac{2(n+2)}{n-2}}([0,\infty); L_x^{\frac{2(n+2)}{n-2}})} {u_m^J} &\leq C(J_0, E_c) < \infty,
\\
\label{pf-badprofile-o/w-4}
\lim_{J\to J^*} \limsup_{m\to\infty} \norm{L_t^2([0,\infty), \dot{W}_x^{1;\frac{2n}{n+2},2})} {e_m^J} &= 0
\end{align}
where $e_m^J = (i\partial_t + \Delta)u_m^J - \mu |x|^{-b} |u_m^J|^\alpha u_m^J$.
Then, we observe by the stability theory that
\[
\norm{L_t^{\frac{2(n+2)}{n-2}} ([0,\infty), L_x^{\frac{2(n+2)}{n-2}})} {u_m} \lesssim C(J_0, E_c) < \infty,
\]
which is a desired contradiction.

\vspace{0.6em}
\eqref{pf-badprofile-o/w-1}, \eqref{pf-badprofile-o/w-2}, and \eqref{pf-badprofile-o/w-3} can be shown similarly as in Killip and Visan \cite{KillipVisan10},
so it remains to prove \eqref{pf-badprofile-o/w-4}.
It suffices to show the two nonlinearity estimates
\begin{align}
\label{nonlest-asympt.1}
\lim_{J\to J^*} \limsup_{m\to\infty}
\norm{L_t^2([0,\infty), L_x^{\frac{2n}{n+2-2b}, 2})}
{\nabla\Big( \sum_{j=1}^{J} F(v_m^j) - F\Big(\sum_{j=1}^{J} v_m^j\Big) \Big)}
&= 0, \\
\label{nonlest-asympt.2}
\lim_{J\to J^*} \limsup_{m\to\infty}
\norm{L_t^2([0,\infty), L_x^{\frac{2n}{n+2-2b}, 2})}
{\nabla\Big( F\Big(\sum_{j=1}^{J} v_m^j\Big) - F(u_m^J) \Big)}
&= 0
\end{align}
where $F(z) = |z|^\alpha z$.
\eqref{nonlest-asympt.1} follows from the asymptotic orthogonality and the sequential smooth approximation \eqref{prevent-disloc-by-inhomog-res-2}.
Unlike \cite[Proposition 3.2]{GuzmanXu24.6}, the high-regularity convergence in \eqref{prevent-disloc-by-inhomog-res-2} is essential to cope with the case $\alpha \leq 1$, since we only have the pointwise estimate
\[
\left| {\nabla\Big( \sum_{j=1}^{J} F(v_m^J) - F\Big(\sum_{j=1}^{J} v_m^J\Big) \Big)} \right|
\lesssim \sum_{j\ne j'} |\nabla v_m^j| |v_m^{j'}|^{\alpha}
\]
and it can happen that $|\nabla v_m^j|$ has to be estimated while $\displaystyle \limsup_{m\to\infty} |{x_m^j}/{\lambda_m^j}| = \infty$.

\vspace{0.6em}
For \eqref{nonlest-asympt.2}, it is enough to show the following cross-term estimate, which follows from \eqref{pf-badprofile-o/w-3} and Strichartz estimate on $\nabla w_m^J$ under the assumption that $\theta=\theta(n,b) \in (0,1)$ is a sufficiently small absolute constant.
\begin{align*}
&\quad \norm{L_t^2([0,\infty), L_x^{\frac{2n}{n+2-2b}, 2})}
{|u_m^J|^\alpha \nabla e^{it\Delta} w_m^J} \\
&\lesssim \norm{L_t^{\frac{n+2}{n-1}}([0,\infty), L_x^{\frac{n+2}{n-1}})}
{u_m^J \nabla e^{it\Delta} w_m^J} ^ \theta
\norm{L_t^{\frac{2(n+2)}{n-2}} ([0,\infty), L_x^{\frac{2(n+2)}{n-2}})}
{u_m^J} ^ {\alpha - \theta}
\norm{L_t^{\frac{2(n+2)}{n-\frac{2(1-b)}{1-\theta}}}([0,\infty), L_x^{\frac{2n(n+2)}{n^2+\frac{4(1-b)}{1-\theta}}, 2})}
{\nabla e^{it\Delta} w_m^J} ^ {1-\theta} \\
&\lesssim \norm{L_t^{\frac{n+2}{n-1}}([0,\infty), L_x^{\frac{n+2}{n-1}})}
{u_m^J \nabla e^{it\Delta} w_m^J} ^ \theta \\
&\lesssim
\norm{L_t^{\frac{n+2}{n-1}}([0,\infty), L_x^{\frac{n+2}{n-1}})}
{\Big(\sum_{j=1}^{J} v_m^j\Big) \nabla e^{it\Delta} w_m^J} ^ \theta +
\norm{L_t^{\frac{n+2}{n-1}}([0,\infty), L_x^{\frac{n+2}{n-1}})}
{e^{it\Delta} w_m^J \nabla e^{it\Delta} w_m^J} ^ \theta
\end{align*}
The remaining steps for showing \eqref{pf-badprofile-o/w-4} follow similarly from the proof of \cite[Lemma 3.2]{KillipVisan10}.
This finishes the proof of Claim \ref{clm:badprofile}.

\vspace{0.6em}
Going back to the proof of Proposition {\ref{palaissmale}},
we may rearrange the indices and assume that there exists $1 \leq J_1 \leq J_0$ such that
\begin{align*}
\limsup_{m \to \infty} \norm{L_t^{\frac{2(n+2)}{n-2}} ([0,\infty) \cap I_m^{j}, L_x^{\frac{2(n+2)}{n-2}} )} {v_m^{j}} &= \infty   \text{\quad for $1\leq j \leq J_1$}, \\
\limsup_{m \to \infty} \norm{L_t^{\frac{2(n+2)}{n-2}} ([0,\infty) \cap I_m^{j}, L_x^{\frac{2(n+2)}{n-2}} )} {v_m^{j}} &< \infty   \text{\quad for $j > J_1$}.
\end{align*}
For each $m,k \in \mathbb{N}$, define an integer $j_{m,k}^\circ \in \{1,2,\cdots,J_1\}$ and an interval $I_m^k$ of the form $[0,T]$ so that
\[
\sup_{1\leq j\leq J_1} \norm{L_t^{\frac{2(n+2)}{n-2}} (I_m^k, L_x^{\frac{2(n+2)}{n-2}} )} {v_m^{j}}
=	\norm{L_t^{\frac{2(n+2)}{n-2}} (I_m^k, L_x^{\frac{2(n+2)}{n-2}} )} {v_m^{j_{m,k}^\circ}}
=	k.
\]
Then, there should exist $j_1 \in \mathbb{N}$ such that $j_{m,k}^\circ = j_1$ for infinitely many $m$ by the pigeonhole principle.
We may rearrange the indices and assume that $j_1 = 1$, so that
\[
\limsup_{m,k \to \infty} \norm{L_t^{\frac{2(n+2)}{n-2}} (I_m^k, L_x^{\frac{2(n+2)}{n-2}} )} {v_m^{1}} = \infty,
\]
and hence, by the definition of $E_c$,
\begin{equation}
\label{Dvm1-seq-Ec}
\limsup_{m,k \to \infty} \sup_{t \in I_m^k} \norm{L_x^2} {\nabla v_m^{1}(t)} ^2 \geq E_c.
\end{equation}

Meanwhile, we notice that all $v_m^j$ have finite scattering sizes on $I_m^k$ for each $k \in \mathbb{N}$ and hence
\begin{equation}
\label{umJ-um-approx-locintime}
\lim_{J\to J^*} \limsup_{m\to\infty}
\norm{L_t^\infty(I_m^k, \dot{H}_x^1)} {u_m^J - u_m} = 0
\end{equation}
for each $k \in \mathbb{N}$ by revisiting the proof of Claim \ref{clm:badprofile}.

\vspace{0.6em}
We would like to recall the kinetic energy decoupling of the approximate solution $u_m^J$,
which is very similar to \cite[Lemma 3.3]{KillipVisan10}. The weak vanishing of the remainder (Proposition \ref{linear-prof-decomp}), (iii) is essential for the proof.
The readers can also see other sources such as \cite[Lemma 5.10]{KillipVisan13}, \cite[p. 246]{KochTataruVisan14}, and \cite[Lemma 3.6]{GuzmanXu24.6}.

\begin{clm} [Kinetic energy decoupling for $u_m^J$] \label{clm:KEdecouple-umJ} ~
\[
\limsup_{m\to\infty} \sup_{t\in I_m^k}
\left| \norm{L_x^2}{\nabla u_m^J(t)}^2 - \sum_{j=1}^{J} \norm{L_x^2}{\nabla v_m^j(t)}^2 - \norm{L_x^2}{\nabla w_m^J}^2 \right|
= 0.
\]
for every $J, k \in \mathbb{N}$.
\end{clm}

Taking all of \eqref{palaissmale-H1-bddness}, \eqref{Dvm1-seq-Ec}, \eqref{umJ-um-approx-locintime}, and Claim \ref{clm:KEdecouple-umJ} into consideration, we see that
\begin{align*}
\limsup_{m\to\infty} \sup_{t\in I_m^k}
\left( \sum_{j=1}^{J} \norm{L_x^2}{\nabla v_m^j(t)}^2 + \norm{L_x^2}{\nabla w_m^J}^2 \right)
&=	\limsup_{m\to\infty} \sup_{t\in I_m^k}
\norm{L_x^2}{\nabla u_m^J(t)}^2 \\
&\leq	E_c
\leq	\limsup_{m,k \to \infty} \sup_{t \in I_m^k} \norm{L_x^2} {\nabla v_m^{1}(t)} ^2
\end{align*}
and hence conclude that $J_1 = 1$ is the only possibility, with $v_m^j \equiv 0$ for every $j \geq 2$ and $w_m^J \to 0$ strongly in $\dot{H}^1(\mathbb{R}^n)$. We rewrite the linear profile decomposition of $u_m(0)$ as
\[
u_m(0) = g_m^1[e^{i t_m^1 \Delta} \phi^1] + w_m^1.
\]

\vspace{0.6em}
As the last step, the $\dot{H}^1(\mathbb{R}^n)$ compactness of the sequence $u_m(0)$ is proved by claiming that $(t_m^1, x_m^1) \equiv (0, 0)$ is the only possibility and therefore $u_m(0) \to g_m^1[\phi^1]$ strongly in $\dot{H}^1(\mathbb{R}^n)$.
For a rough sketch, $|x_m|/{\lambda_m} \to \infty$ is precluded by Proposition \ref{scattering-step-2} and $t_m^1 \to \pm \infty$ is precluded by the monotone convergence, Strichartz estimates, and the vanishing of the remainder.
\end{proof}

We now refine the Palais-Smale condition for the compactness of the minimal energy blow-up solution.
See also \cite[Propositions 4.1, 4.2]{KenigMerle} and \cite[Theorem 1.16]{KillipVisan10}.

\begin{prop}[Minimal energy blow-up solution] \label{compactness-of-orbit} ~

Suppose the failure of Theorem \ref{mainthm}.
Then there is a critical value $0 < E_c < \norm{L^2}{\nabla W}^2$ and a forward-in-time maximal-lifespan solution $u_c \in C_{t,{\rm loc}}([0, T_*^+), \dot{H}_x^1(\mathbb{R}^n))$ to \eqref{INLS} such that
\begin{align*}
\norm{L_t^\infty ([0,T_*^+), L_x^2)}{\nabla u_c} &= E_c < \norm{L^2}{\nabla W}^2, \\
\norm{L_t^{\frac{2(n+2)}{n-2}} ([0,T_*^+), L_x^{\frac{2(n+2)}{n-2}})}{u_c} &= \infty,
\end{align*}
and there is the scaling function $\lambda : [0,T_*^+) \to (0, \infty)$ such that the orbit
\[
K = \left\{ \lambda(t)^{-\frac{n-2}{2}} u_c \Big(t, \dfrac{x}{\lambda(t)}\Big) : 0 \leq t < T_*^+ \right\}
\]
is precompact in $\dot{H}_x^1(\mathbb{R}^n)$. We can also obtain a similar result backward in time.
\end{prop}

\section{Proof of Theorem \ref{mainthm}} \label{Section:4}

In the last section, with the main ingredients ready, we prove Theorem \ref{mainthm} by contradiction.
Any minimal energy blow-up solution ${u_c} : [0, T_*^+) \times \mathbb{R}^n \to \mathbb{C}$ to \eqref{INLS} with the compactness property described in Proposition \ref{compactness-of-orbit} will be eliminated.
We separate the proof into the finite-time blow-up case ($T_*^+ < \infty$) and the soliton case ($T_*^+ = \infty$).

\begin{prop}[No finite-time blow-up] \label{no-finitelife}~

There are no minimal energy blow-up solutions ${u_c}$ with $T_*^+ < \infty$ satisfying Proposition \ref{compactness-of-orbit}.
\end{prop}

\begin{proof}
We use the following lemma introduced by \cite[Section 6]{TaoVisanZhang08}. (See also \cite[Proposition 5.23]{KillipVisan13}.)
\begin{lem} [Reduced Duhamel formula] \label{redduhamel} ~

For each $t \in [0, T_*^+)$, the following weak convergence holds as $T \nearrow T_*^+$.
\[
{u_c}(t) + i\mu \int_{t}^{T} e^{i(t-\tau)\Delta} \big( |x|^{-b} |{u_c}|^\alpha {u_c} \big)(\tau) \,d\tau \xrightharpoonup{\quad} 0
\text{\quad in $\dot{H}_x^1(\mathbb{R}^d)$}
\]
\end{lem}
\noindent
Combined with Bernstein's inequality, Lemma \ref{redduhamel} implies ${u_c} \in C_t([0,T_*^+), L_x^2(\mathbb{R}^n))$ with
\begin{align*}
\norm{L_x^2}{{u_c}(t)}
&\leq	\norm{L_x^2} {P_{\leq M} {u_c}(t)} + \norm{L_x^2} {P_{>M} {u_c}(t)} \\
&\lesssim	\norm{L_t^1([t,T_*^+), L_x^2)} {P_{\leq M} \big( |x|^{-b} |{u_c}|^\alpha {u_c} \big)} + M^{-1} \norm{L_x^2} {\nabla {u_c}(t)} \\
&\lesssim	M (T_*^+ - t) \norm{L_t^\infty([0,T_*^+), L_x^{\frac{2n}{n+2}, 2})} {|x|^{-b} |{u_c}|^\alpha {u_c}} + M^{-1} E_c \\
&\lesssim	M (T_*^+ - t) E_c^{\alpha + 1} + M^{-1} E_c.
\end{align*}
for every time $t \in [0,T_*^+)$ and frequency $M \in 2^\mathbb{Z}$.
Therefore, by the conservation of mass, taking $M$ arbitrarily large and $t \nearrow T_*^+$, we conclude that $M[{u_c}] = 0$ and hence ${u_c} \equiv 0$, which is a contradiction.
Compared to \cite[Subsection 3.1 Claim 1]{GuzmanXu24.6}, the use of Lorentz spaces relieves the need for Hardy inequality, and we no longer observe the additional range restrictions to $n$ and $b$.
\end{proof}

\begin{prop}[No soliton] \label{no-soliton}~

There are no minimal energy blow-up solutions ${u_c}$ with $T_*^+ = \infty$ satisfying Proposition \ref{compactness-of-orbit}.
\end{prop}

\begin{proof}
We revisit the proof of \cite[Subsection 3.1 Claim 2]{GuzmanXu24.6} for completeness. We can see by direct computation that the argument works for every $n\geq 3$ and $0<b<\min(2,n/2)$.
We start with recalling the Morawetz identity as appears in several previous studies for \eqref{INLS}.
For any solution $u(t,x)$ to \eqref{INLS},
given the quantity
\[
Z(t) = 2\,\mathrm{Im} \int_{\mathbb{R}^n} \bar{u}\nabla u \cdot \nabla a ~dx
\]
where $a : \mathbb{R}^n \to \mathbb{R}$ is any smooth function with compact support,
we compute and observe the identity
\begin{align*}
\partial_t Z(t)
&=	4 \sum_{j,k=1}^{n} \int_{\mathbb{R}^n} \mathrm{Re}(\overline{\partial_{x_j}u} {\partial_{x_k}u}) \partial_{x_j}\partial_{x_k}a ~dx
- \int_{\mathbb{R}^n} |u|^2 \Delta^2 a ~dx \\
&\quad	+ \frac{2\alpha}{\alpha+2} \int_{\mathbb{R}^n} \mu |x|^{-b} |u|^{\alpha+2} \Delta a ~dx
- \frac{4}{\alpha+2} \int_{\mathbb{R}^n} |u|^{\alpha+2} (-\mu b|x|^{-b-2}x) \cdot \nabla a ~dx.
\end{align*}
Choose $a(x) = R^2 \psi\Big(\dfrac{|x|}{R}\Big)$, where $\psi : [0, \infty) \to [0, \infty)$ is a smooth function such that $\psi(r) = r^2$ for $r \in [0, 1]$ and $\psi(r) = 0$ for $r \in [2, \infty)$.
Take $u = u_c$. After further computation, we observe
\begin{align*}
\partial_t Z(t)
=&\;	8 \int_{\mathbb{R}^n} |\nabla u_c|^2 + \mu |x|^{-b} |u_c|^{\alpha+2} ~dx \\
&	+ O\left( \int_{|x| > R} |\nabla u_c|^2 ~dx + |x|^{-b} |u_c|^{\alpha+2} + |x|^{-2}|u_c|^2 ~dx \right).
\end{align*}
By the compactness property in Proposition \ref{compactness-of-orbit},
we can deduce that the big-$O$ terms in the right hand side vanish uniformly in $t$ as $R \to \infty$, which is analogous to \cite[Lemma 5.5]{KenigMerle}. Thus, for every $t \in [0,\infty)$, we have by the variational analysis that
\[
\partial_t Z(t) \geq \delta \int_{\mathbb{R}^n} |\nabla u_c(t)|^2 dx - o(1) \gtrsim \delta E[u_c] - o(1)
\]
where $o(1) \to 0$ as $R \to \infty$.
Meanwhile, we see also from H\"older's inequality that
\[
\sup_{t\in[0,\infty)} |Z(t)|
\lesssim	R^2 \norm{L_t^\infty ([0,\infty), L_x^2)} {\nabla u_c}
\sim	R^2 E[u_c],
\]
and thus the Fundamental Theorem of Calculus to $Z(t)$ should imply
\[
\big(\delta E[u_c] - o(1)\big)T
\lesssim \int_{0}^{T} \partial_t Z(t) dt
\leq |Z(T)| + |Z(0)|
\lesssim R^2 E[u_c]
\]
for every $0 < R,T < \infty$.
However, fixing $0<R<\infty$ large enough, the left-hand side grows indefinitely as $T \to \infty$, which is a contradiction.

\end{proof}

\bibliographystyle{plain}
\bibliography{INLS-scatter-foc-ref}

\end{document}